\documentclass[11pt,oneside,reqno,british]{amsart}
\usepackage[T1]{fontenc}
\usepackage[latin9]{inputenc}
\usepackage{color}
\usepackage{babel}
\usepackage{mathtools}
\usepackage{amstext}
\usepackage{amsthm}
\usepackage{amssymb}
\usepackage[unicode=true,pdfusetitle,
 bookmarks=true,bookmarksnumbered=false,bookmarksopen=false,
 breaklinks=true,pdfborder={0 0 0},pdfborderstyle={},backref=false,colorlinks=true]
 {hyperref}
\usepackage{tikz}
\usetikzlibrary{circuits.ee.IEC}

\makeatletter
\theoremstyle{definition}
\newtheorem*{defn*}{Definition}
\newtheorem*{conj*}{Conjecture}
\newtheorem{conj}{Conjecture}
\newtheorem{quest}[conj]{Question}

\theoremstyle{plain}
\newtheorem{thm}{Theorem}
\newtheorem*{thm*}{Theorem}
\newtheorem{lem}{Lemma}

\newtheorem*{claim*}{Claim}

\theoremstyle{remark}
\newtheorem*{rem*}{Remark}
\newtheorem*{rmrks*}{Remarks}

\usepackage{calrsfs}
\usepackage{graphicx}
\usepackage{color}
\usepackage{perpage}
\usepackage{upquote}
\usepackage{bbm}
\usepackage[shortlabels]{enumitem}
\usepackage{stmaryrd} 
\usepackage{extarrows} 
\usepackage{romanbar}
\usepackage{color}

\allowdisplaybreaks
\MakePerPage{footnote}
\gdef\SetFigFontNFSS#1#2#3#4#5{} 
\gdef\SetFigFont#1#2#3#4#5{} 
\def\clap#1{\hbox to 0pt{\hss#1\hss}}

\DeclareMathOperator{\sign}{sign}

\definecolor{myblue}{rgb}{0.09,0.32,0.44} 
\hypersetup{pdfborder={0 0 0},pdfborderstyle={},colorlinks=true,linkcolor=myblue,citecolor=myblue,urlcolor=blue}

\newlength{\tempindent}
\newcommand{\lazyenum}{
\setlength{\tempindent}{\parindent}
\begin{enumerate}[leftmargin=0cm,itemindent=0.7cm,labelwidth=\itemindent,labelsep=0cm,align=left,label=\arabic*)]
\setlength{\parskip}{\smallskipamount}
\setlength{\parindent}{\tempindent}
}
\newcommand{\enumrom}{
  \begin{enumerate}[label=(\roman*)]
}
\makeatletter
\renewcommand{\andify}{%
  \nxandlist{\unskip, }{\unskip{} \@@and~}{\unskip{} \@@and~}}
\def\author@andify{%
  \nxandlist {\unskip ,\penalty-1 \space\ignorespaces}%
    {\unskip {} \@@and~}%
    {\unskip \penalty-2 \space \@@and~}%
}
\let\@wraptoccontribs\wraptoccontribs
\makeatother

\makeatletter
\def\moverlay{\mathpalette\mov@rlay}
\def\mov@rlay#1#2{\leavevmode\vtop{%
   \baselineskip\z@skip \lineskiplimit-\maxdimen
   \ialign{\hfil$\m@th#1##$\hfil\cr#2\crcr}}}
\newcommand{\charfusion}[3][\mathord]{
    #1{\ifx#1\mathop\vphantom{#2}\fi
        \mathpalette\mov@rlay{#2\cr#3}
      }
    \ifx#1\mathop\expandafter\displaylimits\fi}
\makeatother

\ifpdf
\def\afs#1#2{\href{#1}{\nolinkurl{#2}}}
\else
\def\afs#1#2{\href{#1}{\nolinkurl{#2}}}
\fi

\def\expand#1#2{#1\langle #2\rangle}
\def\expandp#1#2{\expand{#1}{#2}}

\def\affs#1#2{\href{#1}{\nolinkurl{#2}}}

\def\lr{\leftrightarrow}
\def\nlr{\nleftrightarrow}
\def\Bl{\mathcal{B}}
\def\Cl{\mathcal{C}}

\def\El{\mathcal{E}}
\def\Gl{\mathcal{G}}
\def\Hl{\mathcal{H}}
\def\Pl{\mathcal{P}}
\def\Ul{\mathcal{U}}
\def\Wl{\mathcal{W}}
\def\Xl{\mathcal{X}}

\def\PP{\mathbb{P}}
\def\EE{\mathbb{E}}
\def\NN{\mathbb{N}}
\def\RR{\mathbb{R}}
\def\ZZ{\mathbb{Z}}
\def\eps{\varepsilon}

\def\piv{\partial_{\mathrm{iv}}}
\def\pev{\partial_{\mathrm{ev}}}

\newcommand\fullsizeground{\begin{tikzpicture}[baseline]
  \draw (0,0.27) -- (0,0.10);%
  \draw (-0.12,0.10) -- (0.12,0.10);%
  \draw (-0.08,0.05) -- (0.08,0.05);%
  \draw (-0.04,0) -- (0.04,0);%
  \end{tikzpicture}%
}
\newcommand\Ground{\mathbin{%
  \mathchoice{\fullsizeground}{\fullsizeground}%
    {\scalebox{0.65}{\fullsizeground}}%
    {\scalebox{0.4}{\fullsizeground}}}}

\makeatother

\title[A precolation inequality]{A reduction of the $\theta(p_c) = 0$ problem to a conjectured inequality}

\author{Gady Kozma$^1$}
\address{GK: The Weizmann Institute of Science, Rehovot, Israel}
\email{gady.kozma@weizmann.ac.il}
\thanks{\textcolor{white}{l}$^1$ The first author is partially  supported by the Israel Science Foundation and the Jesselson Foundation.}

\author{Shahaf Nitzan$^2$}
\address{SN: Georgia Institute of Technology, Atlanta, USA}
\email{shahaf.nitzan@math.gatech.edu}
\thanks{\textcolor{white}{l}$^2$ The second author is supported by NSF CAREER grant DMS 1847796.}

\begin{document}

\begin{abstract}
We conjecture a new correlation-like inequality for percolation probabilities and support our conjecture with numerical evidence and a few special cases which we prove. This inequality, if true, implies that there is no percolation at criticality at $\ZZ^d$, for any $d>1$.
\end{abstract}

\maketitle

\section{Introduction}

Consider the graph $\ZZ^d$ (precise definitions are given below in \S\ref{sec:prelim}). Let $p\in[0,1]$ be some probability and let $\omega$ be the random subgraph of $\ZZ^d$ one gets by removing each edge with probability $1-p$ independently. Percolation theory starts with the question: For which $p$ does the random graph $\omega$ have an infinite connected component? (an infinite cluster, in the common terminology). See the book \cite{Gri} for a comprehensive introduction.
It is well-known and easy to prove that there exists a number, denoted by $p_c$, such that for $p<p_c$ the graph $\omega$ almost surely has only finite components, while for $p>p_c$ there exists an infinite component, almost surely. A captivating question, open since at least the 80s, is: Does an infinite cluster exist at the exact value of $p_c$, for $d\ge 2$? It is common to denote
\[
\theta(p)\coloneqq\PP(\textrm{the cluster of 0 is infinite}).
\]
It is well-known that this question is equivalent to asking whether $\theta(p_c)=0$. A second equivalent formulation is to ask whether $\theta$ is continuous \cite[\S 8.3]{Gri}. 

The question was resolved for $d=2$ and for $d$ sufficiently large. The result for $d=2$ is due to Harris \cite{H60} who showed that $\theta(\frac12)=0$ and Kesten \cite{K80} who showed that $p_c\le \frac12$. As for large $d$, Hara \cite{H90} showed that $\theta(p_c)=0$ for $d\ge 19$. This was later improved to $d\ge 11$ \cite{FvdH17}. The technique involved is known as `lace expansion' and is inherently limited to $d>6$. See also \cite{BLPS, H, HH} for results on Cayley graphs.

As the problem is quite popular, a number of possible approaches have been suggested to attack it. Many of them have never been published, and are folklore among researchers working on the problem. Therefore it makes sense to list some of these approaches here. 

\lazyenum
\item \label{enu:BGN} 
  The most popular scheme for attacking the problem is hidden within the one-step renormalisation scheme of Barsky-Grimmett-Newman \cite{BGN} and Grimmett-Marstrand \cite{GM}. 
  In both cases, one starts by assuming that $\theta(p)>0$, gets some finite $n$ and some event $E$ defined on the box $[-n,n]^d$ such that $\PP(E)>0.9$. If such an $E$ could be found with the property that $\PP(E)>0.9$ implies $\theta(p)>0$, the problem would be solved, as this would make the set $\{p:\theta(p)>0\}$ open. Such an $E$ is sometimes called a \emph{finite size criterion}.

  In both cases, to justify the second step, namely to show that $\PP(E)$ large implies that $\theta(p)>0$, there is a certain inequality that needs to be justified. In \cite{BGN} it is justified because they work in a half space, and hence the result that they achieve is that the probability of an infinite cluster at $p_c$ \emph{in a half-space} is zero. In \cite{GM} it is justified by moving to a larger $p$ and a \emph{sprinkling argument}, and hence their result is about $p>p_c$ and not about $p_c$ itself (see a short discussion in \cite{Gri}, before lemma 7.17). It is well-known that if this inequality can be justified without either assumption then the conjecture that $\theta(p_c)=0$ would follow. This approach will be described in greater detail below.
\item Newman showed that if $\theta(p_c)>0$ then
  \[
  \int_0^{p_c}\EE_p(|\Cl(0)|)^{1/2}\,dp=\infty.
  \]
  See \cite[Theorem 1]{N86}. This gives a possible attack vector for the problem, by showing that this integral is in fact finite. 
\item Duminil-Copin noticed that $\theta(p_c)=0$ is equivalent to showing the claim that
  \[
  \lim_{x\to\infty}\PP_p(0\lr x,0\nlr\infty)= 0,
  \]
  uniformly in $p$. (It is straightforward to show that the limit is 0 for every $p$, so uniformity here is the whole point). 
\item Assume $\theta(p)>0$. A simple ergodicity argument shows that there are infinitely many points in the intersection of the infinite cluster with a horizontal line (say, in the first coordinate and containing $0$). Let $x$ be the point of this intersection closest to zero. In a workshop dedicated to the $\theta(p_c)$ problem held in les Diablerets in 2012, it was noticed that if $\EE(|x|)<\infty$ then $\theta(p_c)=0$. The proof (which, we believe, was never published) is a simple corollary from the Barsky-Grimmett-Newman result \cite{BGN}.
  \item If the two arms exponent (in the sense of \cite{Cerf}, i.e., two disjoint clusters) is too large, then this implies that $\theta(p_c)=0$. See \cite[\S 4.3]{Engel} for some results and historical discussion.
\end{enumerate}
In this paper we return to the one-step renormalisation approach (approach \ref{enu:BGN} above) 
and study it. When we sketched it above, we did not define the `inequality that needs to be justified' because, to the best of our knowledge, it was never explicitly stated in the literature. Mostly, it was considered as an inequality that should be justified \emph{in a specific context}, namely justify that certain connection probabilities of specific sets in $\ZZ^d$ satisfy certain relations.

Our goal in this paper is to give a graph point of view for the same approach. Namely, we will suggest (several possible) inequalities which we believe hold for \emph{all} graphs, and not just in the Euclidean setting of the one-step renormalisation scheme. We will then show that if either of these inequalities hold then they can be incorporated into the one-step renormalisation approach, to give  $\theta(p_c)=0$. 
The inequality which we find cleanest is formulated in the following conjecture.\footnote{After distributing initial drafts of this paper we were informed that J.~van den Berg and D.~Engelenburg independently considered the very same inequality. Their motivation was quite different, it came from attempts to improve the two arms exponent via the approach of \cite{Cerf}. They did not prove or disprove the conjecture.}

\begin{conj}\label{conj:postFKG}Let $G$ be a finite graph with arbitrary probabilities on its edges, let $0, b\in G$ be two vertices and let $A\subset G$ be a set of vertices. Then
  \begin{equation}\label{eq:postFKG}
    \PP(0\lr b)\ge \PP(0\lr A)\min\{\PP(a\lr b):a\in A\}.
  \end{equation}
\end{conj}
We were not able to prove or disprove this conjecture. Our belief that it holds is based on some (admittedly restricted) numerical evidence and on some simple cases where we were able to prove it. In \S\ref{sec:simple cases} we show the cases of conjecture \ref{conj:postFKG} which we were able to prove, namely, the case $|A|=2$, some particular cases in which $|A|=3$, and graphs where $0$ is `very close' to $A$. In \S\ref{sec:theta} we show that conjecture \ref{conj:postFKG} implies that $\theta(p_c)=0$. Finally \S\ref{sec:misc} contains various musings, remarks, possible and impossible generalisations and the like.

Before we start with all this, let us note two additional version of conjecture \ref{conj:postFKG} which we find particularly important. One reason to believe that conjecture \ref{conj:postFKG} holds is that it has an `FKG-like' feeling. Nevertheless, in fact we believe the following stronger inequality, which does not have this property, might also be true.
\begin{conj}\label{conj:preFKG}Let $G$, $0$, $b$ and $A$ be as in conjecture \ref{conj:postFKG}. Then
  \begin{equation}\label{eq:preFKG}
    \PP(0\lr b)\ge \min\{\PP(0\lr A,a\lr b):a\in A\}.
  \end{equation}
\end{conj}
Conjecture \ref{conj:postFKG} follows from conjecture \ref{conj:preFKG} by an application of the FKG inequality. For this reason we sometimes call conjecture \ref{conj:postFKG} `the post-FKG conjecture' and conjecture \ref{conj:preFKG} `the pre-FKG conjecture'.

We will also occasionally prove the following formally stronger version of conjecture \ref{conj:preFKG}:
\begin{equation}\label{eq:preFKGA}
  \PP(0\lr b,0\lr A)\ge \min\{\PP(0\lr A,a\lr b):a\in A\}.
\end{equation}
We say that \eqref{eq:preFKGA} is only `formally' stronger than \eqref{eq:preFKG} because to prove \eqref{eq:preFKG} it is enough to prove it in the case $b\in A$ (the case that $b\not\in A$ follows by considering \eqref{eq:preFKG} for $A'=A\cup\{b\}$), and in this case \eqref{eq:preFKGA} and \eqref{eq:preFKG} are equivalent.

\section{Preliminaries}\label{sec:prelim}

Below we define some notation, and discuss a correlation inequality due to van den Berg, H\"aggstr\"om and Kahn.

\subsection{Notation}

Throughout, we will not distinguish between a graph and its set of vertices. Thus if $G$ is a graph, the notation $v\in G$ stands for `$v$ is a vertex of $G$' and $A\subseteq G$ stands for `$A$ is a subset of vertices'. The set of edges of $G$ will be denoted by $E(G)$. The notation $v\sim w$ will mean that $v$ and $w$ are neighbours in the graph (or $\{v,w\}\in E(G)$).

For a graph $G$ and a function $p:E(G)\to[0,1]$ we will consider the measure on subsets $\omega$ of $E(G)$ such that for every edge $e$ the probability that $e\in\omega$ is $p(e)$, and these events are independent. We call such measures \emph{percolation measures}. Edges which satisfy $e\in\omega$ are called \emph{open}. Otherwise the edges are \emph{closed}. For two vertices $v$ and $w$ we denote by $v\lr_\omega w$ the event that $v$ and $w$ are connected in $\omega$, namely, that there exist $v=x_1\sim x_2\sim\dotsb\sim x_k=w$ such that $\{x_i,x_{i+1}\}\in \omega$ for all $i\in\{1,\dotsc,k-1\}$. For a vertex $v$ we denote by $\Cl_\omega(v)$ the \emph{cluster} of $v$ in $\omega$, namely $\{w:v\lr_\omega w\}$. For a set of vertices $A$ we denote $\Cl(A)=\bigcup_{a\in A}\Cl(a)$. We denote by $v\lr_\omega\infty$ the event $|\Cl_\omega(v)|=\infty$. When $\omega$ is random and given by a percolation measure we will usually omit it from the notation. Thus $\PP(v\lr w)$ is the probability that $v\lr_\omega w$ for $\omega$ given by some percolation measure, which is assumed to be clear from the context (sometimes we will denote it by $\PP_G$ for clarity, where $G$ is the graph). For a number $p\in[0,1]$ the notation $\PP_p$ will be used for the percolation measure with the constant function $p$.

When $v,w\in A\subset G$ we denote by $v\stackrel{A}{\lr}w$ the event that there exist $x_i$ as above with $x_i\in A$ for all $i$. The set $\{w:v\stackrel{A}\lr w\}$ will be called \emph{the cluster of $v$ in $A$}. For $A\subset B$ we denote by $\mathcal{C}(A;B)=\{x\in B:\exists y\in A\text{ s.t. }x\xleftrightarrow{B}y\}$ the cluster of $A$ in $B$. 

We will occasionally need site percolation as well. Site percolation on a graph $G$ is a random subset of \emph{vertices} rather than edges. Two vertices $v$, $w$ are considered connected in $\omega$ (which is now a random set of vertices) if there exist $v=x_1\sim\dotsb\sim x_k=w$ such that $x_i\in \omega$ for all $i$. The definition of $\Cl(v)$ is analogous.

We denote by $\ZZ^d$ the graph which has as its vertex set $\ZZ^d$ and as its edge set $\{\{v,w\}:||v-w||_1=1\}$. As above, we use the notation
\[
\theta(p)\coloneqq\PP_p(|\Cl(0)|=\infty).
\]
Since $\theta(p)$ is increasing there exists a value $p_c=p_c(\ZZ^d)\in [0,1]$, called the critical probability, such that for $p<p_c$ one has $\theta(p)=0$ while for $p>p_c$, $\theta(p)>0$. For site percolation on $\ZZ^d$ we denote by $p_c(\ZZ^d,\textrm{site})$ the corresponding critical value. See \cite[Theorem 1.10]{Gri} for the fact that $p_c\ne 0,1$.

For a graph $G$, an increasing function is a function $f:\{0,1\}^{E(G)}\to\RR$ which is increasing in each variable (we will not distinguish between $\{0,1\}^{E(G)}$ and the set of subsets of $E(G)$, considering each element $\omega\in\{0,1\}^{E(G)}$ as the subset $\{e:\omega(e)=1\}$). The Fortuin-Kasteleyn-Ginibre inequality (henceforth abbreviated as FKG, also known as `the Harris inequality') states that for any two increasing functions $f$ and $g$ one has $\EE(fg)\ge \EE(f)\EE(g)$ (under suitable integrability conditions, e.g.\ that both are in $L^2$). See \cite[\S2.2]{Gri}.

Let $G$ be a graph, let $U$ be some event (i.e.\ $U \subseteq\{0,1\}^{E(G)}$) and let $e$ be an edge. The event that \emph{$e$ is pivotal for $U$} is the set of $\omega\subseteq E(G)$ such that $(\omega \cup \{e\})\in U$ but $(\omega\setminus\{e\})\not\in U$, or vice versa. In other words, $e$ is pivotal if changing the status of $e$ changes whether the event $U$ occurred or not. Notice that $e$ being pivotal is an event which does not depend on the status of $e$. If you want, an event on $E(G)\setminus\{e\}$.

\subsection{The van den Berg-H\"aggstr\"om-Kahn inequality}

We will use the titular inequality profusely, so let us start by stating it. Let $G$ be a graph and let $A$ be a set of vertices of $G$. We say that a function $f:\{0,1\}^{E(G)}\to\RR$ is `defined on the cluster of $A$' if for any $\omega,\omega'\subseteq {E(G)}$ for which the cluster of $A$ is identical, i.e.\ $\Cl_\omega(A)=\Cl_{\omega'}(A)$, we have that $f(\omega)=f(\omega')$. 
We say that an event is defined on the cluster of $A$ if its indicator function is defined on the cluster of $A$.

We may now state the inequality (see \cite[theorems 1.2 and 1.4]{vdBHK} for the proof).

\begin{thm*}[van den Berg, H\"{a}ggstr\"om and Kahn (BHK)]\label{rsa}Let $G$ be a finite graph and $A, B\subset G$. Let $f$ and $g$ be two increasing functions defined on the cluster of $A$. Then
  \[
  \EE(fg\,|\, A\nleftrightarrow B) \ge \EE(f\,|\,A\nleftrightarrow B) \EE(g\,|\,A\nleftrightarrow B).
  \]
  The inequality is reversed if either $f$ or $g$ are decreasing (but not both). Contrariwise, if $f$ and $g$ are two increasing functions defined on the clusters of $A$ and $B$ respectively then:
  \[
  \EE(fg\,|\,A\nleftrightarrow B) \le \EE(f\,|\,A\nleftrightarrow B) \EE(g\,|\,A\nleftrightarrow B).
  \]
Again, of course, if either $f$ or $g$ are decreasing then the inequality is reversed.
\end{thm*}

The following lemmas are simple consequences of the above.
\begin{lem}\label{pitzutzimbainaim} Let $G$ be a graph, $0\in G$, $A=\{a_j\}_{j=1}^n\subset G$ and $X\subset\{1,\dotsc,n\}$. Denote by $A(X)=\{a_j\}_{j\in X}$ the corresponding subset of $A$.


  \enumrom
\item If $Q$ is an increasing event defined on the cluster of $A(X)$ or a decreasing event defined on the cluster of $A(X^c)$ then
\[
\PP(0\lr A(X)\,|\,A(X)\nlr A(X^c))\leq \PP(0\lr A(X)\,|\,A(X)\nlr A(X^c), Q)
\]
\item If $Q$ is a decreasing event defined on the cluster of $A(X)$ or an increasing event defined on the cluster of $A(X^c)$ then
\[
\PP(0\lr A(X)\,|\,A(X)\nlr A(X^c), Q)\leq \PP(0\lr A(X)\,|\,A(X)\nlr A(X^c))
\]
\end{enumerate}
\end{lem}
\begin{proof}
We prove part (i):  From BHK we have,
\begin{multline*}
\PP(0\leftrightarrow A(X), A(X)\nleftrightarrow A(X^c))\PP (Q,A(X)\nleftrightarrow A(X^c))\leq\\ \PP(0\leftrightarrow A(X), Q, A(X)\nlr A(X^c))\PP( A(X)\nlr A(X^c)).
\end{multline*}
the claim follows. Part (ii) is proved identically.
\end{proof}

\begin{lem}\label{lem:phi}
With the same notations as in the previous lemma, denote in addition $\phi(X):=\PP(0\lr A(X)\,|\,A(X)\nlr A(X^c))$. For some $X\subset\{1,\dotsc,n\}$, assume a decomposition of $X$, $X=\cup_{k=1}^K X_k$ with the sets $X_k$ being disjoint. Then,
\[
\sum_k\phi(X_k)\leq \phi(X).
\]
\end{lem}
\begin{proof}
Let $M$ be the event defined by
\[
M\coloneqq\{A(X)\nlr A(X^c), A(X_k)\nlr A(X^c_k), \forall k=1,\dotsc,K\}.
\]
For $k=1,\dotsc,K$ we can write $M$ as $M=\{A(X_k)\nlr A(X^c_k)\}\cap Q$ where $Q$ is a decreasing event defined on the cluster of $A(X^c_k)$. By part (i) of the previous lemma we therefore have
\[
\phi(X_k)\leq \PP(0\lr A(X_k)\,|\,M)
\]
and so
\[
\sum_{k=1}^K\phi(X_k)\leq \sum_{k=1}^K\PP(0\lr A(X_k)\,|\,M)=\PP(0\lr A(X)\,|\,M)
\]
where the equality follows because under $M$ these are disjoint events.

On the other hand, we can write $M$ also as $M=\{A(X)\nlr A(X^c)\}\cap Q$ where $Q$ is a decreasing event defined on the cluster of $A(X)$. By part (ii) of the same lemma we therefore get
\[
\PP(0\lr A(X)\,|\,M)\leq \phi(X).
\]
Combining the last two inequalities the result follows.
\end{proof}

\begin{lem}\label{lemma; compare}
Let $G$ be a graph, $a_1,a_2, b\in G$ and $\delta>0$ be such that
\[
\PP(a_1\lr b)<\PP(a_2\lr b)+\delta.
\]
\enumrom
\item If $Q$ is an increasing event in the cluster of $a_2$ then we have
\[
\PP(a_1\lr b\,|\,Q)<\PP(a_2\lr b\,|\,Q)+\delta.
\]
\item If $Q$ is a decreasing event in the cluster of $a_1$ then we have
\[
\PP(a_1\lr b, Q)<\PP(a_2\lr b, Q)+\delta.
\]
\end{enumerate}
\end{lem}
\begin{proof}
We first prove (i). We start by removing $\PP(a_1\lr a_2\lr b)$ from the given inequality, and dividing it by $\PP(a_1\nlr a_2)$. We get
\[
\PP(a_1\lr b\,|\, a_1\nlr a_2)<\PP(a_2\lr b\,|\, a_1\nlr a_2)+\frac{\delta}{\PP(a_1\nlr a_2)}.
\]
Next, we apply lemma \ref{pitzutzimbainaim}(i) to the right hand side and lemma \ref{pitzutzimbainaim}(ii) to the left hand side. In both cases we use the lemma wirth $0_{\textrm{lemma\ref{pitzutzimbainaim}}}=b$ and $A_{\textrm{lemma\ref{pitzutzimbainaim}}}=\{a_1,a_2\}$. We find that
\begin{equation}\label{compare}
\PP(a_1\lr b\,|\, a_1\nlr a_2, Q)<\PP(a_2\lr b\,|\, a_1\nlr a_2,Q)+\frac{\delta}{\PP(a_1\nlr a_2)}.
\end{equation}
We now multiply both sides by $\PP(a_1\nlr a_2\,|\, Q)$ and get
\[
\PP(a_1\lr b, a_1\nlr a_2\,|\, Q)<\PP(a_2\lr b, a_1\nlr a_2\,|\, Q)+\frac{\delta \PP(a_1\nlr a_2\,|\, Q)}{\PP(a_1\nlr a_2)}.
\]
Since $a_1\nlr a_2$ is a decreasing event and $Q$ is an increasing event the FKG inequality allows to bound the rightmost term and we get 
\[
\PP(a_1\lr b, a_1\nlr a_2\,|\, Q)<\PP(a_2\lr b, a_1\nlr a_2\,|\, Q)+\delta.
\]
adding $\PP(a_1\lr a_2\lr b\,|\,Q)$ to both sides completes the proof for part (i). The proof for part (ii) is similar, but the FKG inequality is no longer available for us. So in this case we multiply both sides of \eqref{compare} by $\PP(a_1\nlr a_2, Q)$ (which is clearly smaller then $\PP(a_1\nlr a_2)$). The proof is now completed in the same way as above.
\end{proof}
\section{Some cases where the inequality holds}\label{sec:simple cases}

In this section we discuss some particular cases in which we can show that conjectures \ref{conj:postFKG} and \ref{conj:preFKG} hold. In \S\ref{sec:small A} we focus on cases where $|A|$ is small, and in \S\ref{sec:A close} we discuss cases where $0$ is `close' to $A$.

\subsection{Small $A$}\label{sec:small A}

\begin{thm}\label{thm:A=2}The pre-FKG conjecture \eqref{eq:preFKGA} holds when $|A|=2$.
\end{thm}
\begin{proof}Recall the notation $\phi$ from lemma \ref{lem:phi} and denote for brevity $\phi(1)\coloneqq\phi(\{1\})$ and $\phi(2)\coloneqq\phi(\{2\})$. Denote
  \begin{equation}\label{eq:A=2coeffs}
  \alpha\coloneqq\frac{\phi(1)}{\phi(1)+\phi(2)}\qquad \beta\coloneqq\frac{\phi(2)}{\phi(1)+\phi(2)}.
  \end{equation}
The theorem will follow if we prove that
\begin{equation}\label{eq:thm A=2}
  \PP(0\lr b\lr A)\geq \alpha\PP(0\lr A, a_1\lr b)+\beta \PP(0\lr A, a_2\lr b),
\end{equation}
so we focus on proving \eqref{eq:thm A=2}. Since $\alpha+\beta=1$ this is the same as
\[
(\alpha+\beta)\PP(0\lr b\lr A)\geq \alpha\PP(0\lr A, a_1\lr b)+\beta \PP(0\lr A, a_2\lr b).
\]
Shifting everything to the left, and subtracting some common events from both couples (e.g. $0\lr a_1\lr b$ for the first couple) we get that we need to show
\begin{multline*}
L\coloneqq
\alpha\big(\PP(0\lr a_2\lr b , a_1\nlr a_2)-\PP(0\lr a_2,a_1\lr b, a_1\nlr a_2)\big)+\\
+\beta\big(\PP(0\lr a_1\lr b , a_1\nlr a_2)-\PP(0\lr a_1,a_2\lr b, a_1\nlr a_2)\big)\geq 0.
\end{multline*}
Applying BHK 
4 times gives
\begin{align*}
  L&\ge \alpha\big(\PP(0\lr a_2\,|\, a_1\nlr a_2)\PP(a_2\lr b,a_1\nlr a_2) \\
  &\qquad -\PP(0\lr a_2\,|\,a_1\nlr a_2)\PP(a_1\lr b,a_1\nlr a_2)\big)\\
  &+\;\beta\big(\PP(0\lr a_1\,|\,a_1\nlr a_2)\PP(a_1\lr b,a_1\nlr a_2)\\
  &\qquad -\PP(0\lr a_1\,|\,a_1\nlr a_2)\PP(a_2\lr b,a_1\nlr a_2)\big)\\
  &=\PP(a_2\lr b,a_1\nlr a_2)(\alpha\phi(2)-\beta\phi(1))\\
  &+\; \PP(a_1\lr b,a_1\nlr a_2)(-\alpha\phi(2)+\beta\phi(1))=0,
\end{align*}
where the first equality is simply collecting terms and using the definition of $\phi$, while the second equality is due to our definition of $\alpha$ and $\beta$. Thus \eqref{eq:thm A=2} is proved and so is the theorem.
\end{proof}

\begin{rem*}It is also possible to prove theorem \ref{thm:A=2} as a corollary of lemma \ref{lemma; compare}, but such a proof would not give the coefficients in (\ref{eq:A=2coeffs}), which we find intriguing. We go back to this issue in \S \ref{sec:coef}.
\end{rem*}
We now assume $|A|=3$.
In preparation for theorem \ref{thm:A=3} below, let us first show an easier result. Denote for a $S\subset\{1,2,3\}$
\[
m_S:=\{b\lr a_s\: \forall s\in S, b\nlr A(S^c)\}.
\]
\begin{thm}Assume that for some $j\in\{1,2,3\}$ we have $m_j\leq m_{\{1,2,3\}\setminus j}$. Then the pre-FKG conjecture \eqref{eq:preFKGA} holds.
\end{thm}
\begin{proof}
Again we use the notation $\phi$ from lemma \ref{lem:phi} and again we drop the braces for brevity, e.g.\ $\phi(1,2)\coloneqq\phi(\{1,2\})$.

Without loss of generality we may assume that $\PP(a_3\lr b)\leq \PP(a_j\lr b)$ for $j=1,2$. We now claim that this condition implies that $m_3\leq m_{12}$. Indeed, if the $j$ in the statement of the theorem is 3 this is immediate, and if not we write
\begin{equation}\label{eq:Pm}
  \PP(a_j\lr b)-\PP(a_3\lr b) = (m_j+m_{jk}) - (m_3+m_{3k}),
\end{equation}
where $k$ is the third index (i.e.\ the one different from 3 and $j$). Since the left-hand side of \eqref{eq:Pm} is positive by our assumption and since $m_j\le m_{3k}$ by the assumption of the theorem, we get $m_3\le m_{12}$.

The theorem will follow if we show that $\PP(0\lr b\lr A)\ge \PP(0\lr A,a_3\lr b)$. For this purpose, break $\PP(0\lr b\lr A)$ according to the subset $\Cl(0)\cap A$ and break $\PP(0\lr A,a_3\lr b)$ according to subset $\Cl(b)\cap A$, and for both expansions substract out all terms where $0\lr b\lr a_3$. We get
\[
\PP(0\lr b, 0\lr A)-\PP(0\lr A, a_3\lr b)= I+II+III
\]
where
\begin{align*}
  I&:= \PP(0\lr a_1\lr a_2\lr b, \{a_1,a_2\}\nlr a_3)\\
  &\qquad-\; \PP(0\lr \{a_1, a_2\}, a_3\lr b, \{a_1,a_2\}\nlr a_3)\\
  II&:= \PP(0\lr a_1\lr b, \{a_2,a_3\}\nlr a_1)\\
  &\qquad-\; \PP(0\lr a_1, a_2\lr a_3\lr b, \{a_2,a_3\}\nlr a_1)\\
  III&:= \PP(0\lr a_2\lr b, \{a_1,a_3\}\nlr a_2)\\
  &\qquad-\; \PP(0\lr a_2, a_1\lr a_3\lr b, \{a_1,a_3\}\nlr a_2).
\end{align*}
Applying the van den Berg-H\"aggstr\"om-Kahn inequality to each of these (a total of 6 applications) we find that
\[
I\geq \phi(1,2)(m_{12}-m_3)\qquad
II\geq \phi(1)(m_1-m_{23})
\qquad
III\geq \phi(2)(m_2-m_{13})
\]
Since we know that $m_{12}-m_3\geq 0$, we can use lemma \ref{lem:phi} to give the following estimate for $I$:
\[
I\geq (\phi(1)+\phi(2))(m_{12}-m_3).
\]
We therefore get
\begin{align*}
  \lefteqn{I+II+III}\quad&\\
  &\geq \phi(1)\big( (m_1+m_{12})-(m_3+m_{23})\big)+ \phi(2)\big( (m_2+m_{12})-(m_3+m_{13})\big)\\
  &\stackrel{\eqref{eq:Pm}}{=}
 \phi(1)\big( \PP(a_1\lr b)-\PP(a_3\lr b)\big)+  \phi(2)\big( \PP(a_2\lr b)-\PP(a_3\lr b)\big).
\end{align*}
By our assumption regarding the minimality of $a_3$ the last expression is nonnegative. This completes the proof.
\end{proof}


\begin{lem}\label{lem:piv comp}
Let $G$ be a graph and $a_1,a_2,a_3,b\in G$. Then
\begin{align}
  \max_{j=1,2}\,&\PP(\{a_1,a_2\} \textrm{ is pivotal for } a_j\lr b)\nonumber\\
  \geq\; &\PP(\{a_1,a_2\} \textrm{ is pivotal for }a_3\lr b).\label{eq:piv}
\end{align}
\end{lem}



\begin{proof}

  The lemma is symmetric between $a_1$ and $a_2$ so without loss of generality we may assume that $a_2$ is the one that minimises $\PP(a_j\lr b,a_3\lr \{a_1,a_2\})$. We apply theorem \ref{thm:A=2} with $0_{\textrm{theorem \ref{thm:A=2}}}=a_3$ and get
\[
\PP(a_3\lr b, a_3\lr \{a_1,a_2\})\geq \PP(a_2\lr b, a_3\lr \{a_1,a_2\}).
\]
We subtract $\PP(a_3\lr a_2\lr b)$ from both sides and get
\[
\PP(a_3\lr a_1 \lr b, a_2\nlr a_1)\geq \PP(a_2\lr b, a_1\lr a_3, a_1\nlr a_2).
\]
Certainly, this implies
\[
\PP(a_1\lr b, a_2\nlr \{a_1,a_3\})\geq \PP(a_2\lr b, a_1\lr a_3, a_1\nlr a_2).
\]
We add $\PP(a_1\lr b, a_2\lr a_3, a_1\nlr a_2)$ to both sides and get
\begin{multline*}
  \PP(a_1\lr b, a_2\nlr b)\geq \\
  \PP(a_1\lr b, a_2\lr a_3, a_1\nlr a_2)+ \PP(a_2\lr b, a_1\lr a_3, a_1\nlr a_2).
\end{multline*}
This can be reformulated as  \eqref{eq:piv}, as needed.
\end{proof}

\begin{rem*} For a graph $G$ and an edge $e$ denote by $G\Ground e$ the graph one gets by gluing the two vertices of the edge $e$ (this definition will be used also in \S\ref{sec:misc} in various places). Denote $G^*=G\Ground \{a_1,a_2\}$. It is not difficult to see that \eqref{eq:piv} can be reformulated as
\begin{equation}\label{eq:piv2}
\PP_{G^*}(\{a_1,a_2\}\lr b)- \PP_{G^*}(a_3\lr b)\geq \min_{j=1,2}\PP_{G}(a_j\lr b)- \PP_{G}(a_3\lr b).
\end{equation}
Note that both sides in the above inequality may be either negative or positive.
The minimum on the left hand side is necessary, i.e.\ it is in general not true that
  \[
  \PP_{G^*}(\{a_1,a_2\}\lr b)-\PP_{G^*}(a_3\lr b)\ge \PP_G(a_1\lr b)-\PP_G(a_3\lr b).
  \]
  To see this simply take $a_1=b$ and then the expression becomes $\PP_{G^*}(a_3\lr b)\le \PP_G(a_3\lr b)$ which is almost never correct because the FKG inequality gives the opposite direction.
\end{rem*}

\begin{thm}\label{thm:A=3}If $|A|=3$ and $A$ separates 0 from $b$, namely, any path in $G$ from 0 to $b$ must pass through a point of $A$, then the pre-FKG inequality in the form \eqref{eq:preFKGA} holds.
\end{thm}
\begin{proof}
  Let $G$ be a graph and $A=\{a_1,a_2,a_3\}$ be such that $A$ separates $0$ and $b$. Thus removing $A$ separates $G$ into (at least two) connected components, and $0$ and $b$ are in different components. We denote by $T$ the subgraph of $G$ containing the union of all connected components which do not contain 0, including the vertices in $A$ but not the edges between them. We denote by $B$ the subgraph of $G$ containing the component of $0$, the vertices of $A$ and the edges between them (in fact, components which contain neither 0 nor $b$ could have equally been assigned to $B$, the proof below would have worked the same).

For $1\leq j,k\leq 3$ we will also be interested in the graph $T^{(jk)}=T\Ground \{a_j,a_k\}$ which is obtained from $T$ by gluing the two points $a_j$ and $a_k$. For $X\subset A$ we denote
\[
q_{X}=\PP_T(A(X)\lr b);\qquad q_{X}^{(jk)}=\PP_{T^{jk}}(A(X)\lr b),
\]
where $A(X)$ is as defined in lemma \ref{pitzutzimbainaim}. Here and below we will use $\PP$ for probabilities on $G$, while probabilities on $T$ and $B$ will be denoted by $\PP_T$ and $\PP_B$, respectively.

We assume without loss of generality that $a_3$ minimises the probability of connection to $b$, i.e.\ $\PP(a_3\lr b)\leq \PP(a_j\lr b)$ for $j=1,2$. We will show that under this assumption \eqref{eq:preFKGA} holds for $a_3$, namely 
\[
\PP(0\lr b)\geq \PP(0\lr A, a_3\lr b).
\]
which will prove the theorem.

We start with the condition
\begin{equation}\label{conda3min}
0 \leq \PP(a_1\lr b) - \PP(a_3\lr b).
\end{equation}
We wish to break this inequality according to the connection patterns in $B$. For this purpose let $M$ be the event that $\{a_1\stackrel{B}{\nlr} a_2\}\cap\{a_1\stackrel{B}{\nlr} a_3\}\cap\{a_2\stackrel{B}{\nlr} a_3\}$.
With this notation \eqref{conda3min} can be rewritten as
\begin{multline}\label{conda3min-ii}
  0\leq \PP_B(a_1\lr a_2\nlr a_3)(q_{12}-q^{12}_3)\;+\\
  \PP_B(a_1\nlr a_2\lr a_3)(q_1^{23}-q_{23})+\PP_B(M)(q_1-q_3).
\end{multline}
(Of course, $\PP_B(M)=\PP(M)$). A similar relation holds also for $a_2$.

Next, we note that since $\PP(a_3\lr b)\le\PP(a_1\lr b)$ and $\{a_3\nlr a_2\}$ is a decreasing event in the cluster $a_3$, part (ii) of lemma \ref{lemma; compare} (with $\delta=0$) implies that
\begin{equation}\label{eq:a3nlra2}
\PP(a_3\lr b, a_3\nlr a_2)\leq \PP(a_1\lr b, a_3\nlr a_2)
\end{equation}
We now take \eqref{eq:a3nlra2} and break it according to the connection patterns in $B$. We get
\begin{equation}\label{cond2}
0\leq \PP_B(a_1\lr a_2\nlr a_3)(q_{12}-q^{(12)}_3)+\PP(M)(q_1-q_3).
\end{equation}
A similar relation holds also when $a_1$ and $a_2$ are swapped. With these two relations established we are in a position to apply lemma \ref{lem:piv comp}. We apply it to the graph $T$ and the points $a_1,a_2,a_3,b$. Let $j\in\{1,2\}$ be such that the minimum in \eqref{eq:piv2} is achieved at $j$. In our notations the conclusion of the lemma (as it appears in \eqref{eq:piv2}) thus becomes
\[
q_{12}-q_3^{(12)}\ge q_j-q_3.
\]
We conclude that
\begin{equation}\label{cond1}
q_{12}-q^{(12)}_3\geq 0.
\end{equation}
Indeed, if it were negative, so must be $q_j-q_3$ leading to a contradiction to \eqref{cond2} (if $j=1$) or to \eqref{cond2} with $a_1$ and $a_2$ swapped.

Let us now consider the inequality we wish to prove,
\[
0\leq \PP(0\lr b) - \PP(0\lr A, a_3\lr b).
\]
We decompose it according to the connections between 0 and the elements of $A$ in $B$ to get that it is enough to show
\[
\begin{aligned}
0&\leq \PP_B(0\lr a_1\lr a_2, A(1,2)\nlr a_3)(q_{12}-q^{(12)}_3)\\
&+ \PP_B(0\lr a_1, a_2\lr a_3, a_1\nlr A(2,3))(q^{(23)}_1-q_{23})\\
&+ \PP_B(0\lr a_2, a_1\lr a_3, a_2\nlr A(1,3))(q^{(13)}_2-q_{13})\\
&+ \PP_B(0\lr a_1, M)(q_1-q_3)\\
&+ \PP_B(0\lr a_2, M)(q_2-q_3)\\
&\eqqcolon \alpha+\beta_1+\beta_2+\gamma_1+\gamma_2\\
\end{aligned}
\]
Our first step is to estimate $\alpha$. For this we use lemma \ref{pitzutzimbainaim} twice on the graph $B$. We get
\begin{multline*}
\PP_B(0\lr a_1\lr a_2\,|\, a_1\lr a_2,A(1,2)\nlr a_3)\geq \PP_B(0\lr A(1,2)\,|\, A(1,2)\nlr a_3)\\
\geq \PP_B(0\lr A(1,2)\,|\, M)=\PP_B(0\lr a_1\,|\,M)+\PP_B(0\lr a_2\,|\,M)
\end{multline*}
where the first inequality follows by using clause i) of lemma \ref{pitzutzimbainaim} with $X=\{1,2\}$ and $Q=\{a_1\lr a_2\}$, while the second inequality follows by using clause ii) of lemma \ref{pitzutzimbainaim} with the same $X$ and with $Q=\{a_1\nlr a_2\}$.

Since $q_{12}-q^{(12)}_3>0$ (recall (\ref{cond1})), $\alpha+\gamma_1+\gamma_2$ can therefore be estimated by
\[
  \alpha+\gamma_1+\gamma_2\ge \PP_B(0\lr a_1\,|\,M)L_1 +\PP_B(0\lr a_2\,|\,M)L_2,
\]
where
\[
  \begin{aligned}
  &L_1\coloneqq\PP_B(a_1\lr a_2,A(1,2)\nlr a_3)(q_{12}-q^{(12)}_3)+\PP(M)(q_1-q_3),\\
    &L_2\coloneqq \PP_B(a_1\lr a_2,A(1,2)\nlr a_3)(q_{12}-q^{(12)}_3)+\PP(M)(q_2-q_3).
  \end{aligned}
\]
By (\ref{cond2}) both $L_i$ are positive. As for the multiplicands, we can apply Lemma \ref{pitzutzimbainaim} twice to get
\begin{align*}
\PP_B(0\lr a_1\,|\,M) &\ge \PP_B(0\lr a_1\,|\,a_1\nlr A(2,3))\\
&\ge \PP_B(0\lr a_1\,|\,a_2\lr a_3,a_1\nlr A(2,3))
\end{align*}
where the first inequality follows from clause i) of lemma \ref{pitzutzimbainaim} with $X=\{1\}$ and $Q=\{a_2\nlr a_3\}$ while the second inequality follows from clause ii) of lemma \ref{pitzutzimbainaim} with the same $X$ and with $Q=\{a_2\lr a_3\}$. A similar calculation for $\PP(0\lr a_2\,|\,M)$ gives finally
\begin{align*}
  \alpha+\gamma_1+\gamma_2\geq \;&\PP_B(0\lr a_1\,|\, a_2\lr a_3, a_1\nlr A(2,3))L_1\;+\\
  &\PP_B(0\lr a_2\,|\,a_1\lr a_3, a_2\nlr A(2,3))L_2 \eqqcolon \delta_1+\delta_2
\end{align*}
Condition (\ref{conda3min-ii}) now implies that $\beta_1+\delta_1\geq 0$ and similarly we get that $\beta_2+\delta_2\geq 0$. Thus,
\[
\alpha+\beta_1+\beta_2+\gamma_1+\gamma_2\geq \beta_1+\beta_2+\delta_1+\delta_2\geq 0,
\]
as needed.
\end{proof}

\subsection{$A$ close to 0}\label{sec:A close}

In this section we will show that the pre-FKG conjecture \eqref{eq:preFKGA} holds when in the graph $G\setminus A$ the vertex 0 is isolated, and several simple generalisations. To state the generalisations we will use the following definition (which will only be used in this section): 

\begin{defn*}Let $G$ be a graph, and let $A\subset G$ and $0,b\in G$ be such that $0$ and $b$ are in disjoint components of $G\setminus A$. We say that the quadruple $(G,A,0,b)$ is good if
  \[
  \PP(0\lr b)\ge \min_{a\in A}\PP(a\lr b)-\sum_{W\cap A=\emptyset}\PP(\Cl(0)=W)\min_{a\in A}\PP_{G\setminus W}(a\lr b).
  \]
\end{defn*}

A good graph satisfies the pre-FKG conjecture \eqref{eq:preFKG} because if we denote by $a_0$ the point of $A$ where the minimum of $\PP(a\lr b)$ is achieved then
\begin{align*}
\PP(0\lr b)&\ge \PP(a_0\lr b)-\sum_{W\cap A=\emptyset}\PP(\Cl(0)=W)\min_{a\in A}\PP_{G\setminus W}(a\lr b)\\
&\ge \PP(a_0\lr b)-\sum_{W\cap A=\emptyset}\PP(\Cl(0)=W)\PP_{G\setminus W}(a_0\lr b)\\
&= \PP(a_0\lr b)-\PP(a_0\lr b, 0\nlr A) = \PP(a_0\lr b, 0\lr A).
\end{align*}

Thus the result stated above follows from the following:
\begin{thm}\label{thm:diamond}Let $G$ be a graph and let $b\in G$ and $A\subset G$. If 0 is isolated in $G\setminus A$ then $(G,A,0,b)$ is good.
\end{thm}
We can now state the promised generalisation:
\begin{thm}\label{thm:antenna}Assume that for some $x\in G\setminus A$ we have that $(G\setminus\{0\},A,x,b)$ is good and that 0 is isolated in $G\setminus(A\cup \{x\})$. Then $(G,A,0,b)$ is good.
\end{thm}
Thus \eqref{eq:preFKG} holds, for example, in a graph where 0 is connected only to $A$ and to a certain $x$, which, in turn, is connected only to $A\cup 0$.

We start the proofs of both theorems with a simple corollary of the BHK inequality.
\begin{lem}\label{lem:eps}Let $G$ be a graph, 0 and $b$ some vertices, and $a$ and $v$ two neighbours of 0 with
  \[
  \PP_{G\setminus\{0\}}(a\lr b)\le \PP_{G\setminus\{0\}}(v\lr b).
  \]
  Let $B$ be a set of neighbours of $0$ such that $v\in B$ and let $\sigma\coloneqq\sigma_B$ be the event that all edges from 0 to vertices of $B$ are open and all other edges from 0 are closed. Then
  \[
  \PP(a\lr b,\sigma)\leq \PP(0\lr b,\sigma).
  \]
\end{lem}
\begin{proof} If $a\in B$ then the lemma is obvious (with an equality instead of an inequality) so assume this is not the case. Similarly, we can assume that $\PP_{G\setminus\{0\}}(v\lr a)<1$. There is nothing to prove if $\PP(\sigma)=0$, so assume also this is not the case.

  Next fix some $\eps>0$ and let $H$ be the auxiliary graph (with probabilities) which is identical to $G$ except for the probabilities of edges going out of 0. For these edges we set their probabilities as follows: if the other vertex is in $B$, we set their weight to $\eps$, and otherwise to 0.

  From the requirements of the lemma we now get
  \[
  \PP_H(a\lr b)\le \PP_H(v\lr b)+\eps|B|.
  \]
 Let $E$ be the event that all edges from 0 to vertices of $B$ are open. Note that $E$ is an increasing event in the cluster of $v$.
  Hence by Lemma \ref{lemma; compare}(i) we have
 \[
  \PP_H(a\lr b\,|\,E)\le \PP_H(v\lr b\,|\,E)+\eps|B|.
  \]
  Note that conditioning on $E$ in the graph $H$ is exactly like conditioning on $\sigma$ in $G$, and in particular the two terms of this sort do not depend on $\eps$. Hence we get
  \[
  \PP(a\lr b\,|\,\sigma)\le \PP(v\lr b\,|\,\sigma)+\eps|B|.
  \]
  We take $\eps\to 0$ and note that under $\sigma$ there is no difference between $v\lr b$ and $0\lr b$. Multiplying both sides by $\PP(\sigma)$ proves the lemma.
\end{proof}
\begin{proof}[Proof of theorem \ref{thm:diamond}]
  This is an immediate corollary from lemma \ref{lem:eps}. Indeed, in this case the sum over $W$ in the definition of a good graph contains only one non-zero term, $W=\{0\}$. Denote by $a_0$ the point of $A$ where $\PP_{G\setminus\{0\}}(a\lr b)$ is minimised, and for any $B\subseteq A$ with $B\ne \emptyset$ denote by $\sigma_B$ the event of lemma \ref{lem:eps}. We use lemma \ref{lem:eps} with an arbitrary $v\in B$ and get
  \[
  \PP(0\lr b,\sigma_B )\ge \PP(a_0\lr b,\sigma_B).
  \]
  Summing over all $B\ne \emptyset$ gives
  \begin{align*}
    \PP(0\lr b)&\ge \PP(a_0\lr b,0\lr A) = \PP(a_0\lr b)-\PP(\Cl(0)=\{0\})\PP_{G\setminus\{0\}}(a_0\lr b) \\
    &\ge \min_{a\in A}\PP(a\lr b)-\PP(\Cl(0)=\{0\})\PP_{G\setminus\{0\}}(a_0\lr b),
  \end{align*}
  as needed.
\end{proof}


\begin{proof}[Proof of theorem \ref{thm:antenna}]
  Let $a_0$ be the element of $A$ minimising $\PP_{G\setminus\{0\}}(a\lr b)$. Let $B\subseteq A\cup\{x\}$ with $B\ne \emptyset$ and $\sigma_B$ the event from lemma \ref{lem:eps}. If $B\ne\{x\}$ we can use lemma \ref{lem:eps} with an arbitrary $v\in B\cap A$ and get
  \begin{equation}\label{eq:lemeps}
  \PP(0\lr b,\sigma_B)\ge \PP(a_0\lr b,\sigma_B).
  \end{equation}
  We are left with the case $B=\{x\}$. But in this case we have
  \[
  \PP(0\lr b, \sigma_{\{x\}}) = \PP(\sigma_{\{x\}})\PP_{G\setminus\{0\}}(x\lr b).
  \]
  For the last term we use our assumption that $(G\setminus\{0\},A,x,b)$ is good to write
  \begin{multline*}
    \PP_{G\setminus\{0\}}(x\lr b)\ge \\
    \min_{a\in A}\PP_{G\setminus\{0\}}(a\lr b) -
  \sum_{W\cap A=\emptyset}\PP_{G\setminus\{0\}}(\Cl(x)=W)\min_{a\in A}\PP_{(G\setminus\{0\})\setminus W}(a\lr b).
  \end{multline*}
 Multiplying by $\PP(\sigma_{\{x\}})$ both terms on the right-hand side can be represented as probabilities in $G$ as follows:
  \begin{multline*}
    \PP_{G\setminus\{0\}}(x\lr b)\cdot \PP(\sigma_{\{x\}})\ge\\
    \min_{a\in A}\PP(a\lr b,\sigma_{\{x\}})-
  \sum_{\{0\}\ne W,W\cap A=\emptyset}\PP(\Cl(0)=W)\min_{a\in A}\PP_{G\setminus W}(a\lr b)
  \end{multline*}
  (note that we used the assumption that $A\cup\{x\}$ isolates 0 to claim that any $W\ne\{0\}$, $W\cap A=\emptyset$ for which $\PP(\Cl(0)=W)>0$ must contain $x$). Note that there is no big difference between conditioning on $\sigma_{\{x\}}$ and simply deleting 0. The only difference is the edge between $x$ and $0$, which does not matter for connections between $A$ and $b$. So the first term on the right-hand side of the last equation is simply $\PP(a_0\lr b,\sigma_{\{x\}})$.

  Adding the terms $\PP(0\lr b, \sigma_B)$ for the other $B$ and using \eqref{eq:lemeps} gives
  \begin{align*}
  \PP(0\lr b)&=\sum_{B\ne \emptyset}\PP(0\lr b,\sigma_B) \\
  &\ge \sum_{B\ne \emptyset}\PP(a_0\lr b,\sigma_B)-
  \mkern-24mu\sum_{\{0\}\ne W,W\cap A=\emptyset}\mkern-24mu\PP(\Cl(0)=W)\min_{a\in A}\PP_{G\setminus W}(a\lr b)\\
  &=\PP(a_0\lr b)-
  \sum_{W\cap A=\emptyset}\PP(\Cl(0)=W)\min_{a\in A}\PP_{G\setminus W}(a\lr b),
  \end{align*}
  as needed.
\end{proof}

\section{$\theta(p_c)$}\label{sec:theta}

In this section we show that conjecture \ref{conj:postFKG} implies that $\theta(p_c)=0$ in all dimensions. Since even the wise cannot see all ends, we will base our proof on a weaker version of conjecture \ref{conj:postFKG}, as follows.

\begin{conj}\label{conj:epsdel} For every $\eps>0$ there exists $\delta>0$ such that for any graph $G$, any $A\subset G$, and any $0,b\in G$ which satisfy $\PP(0\lr A)> 1-\delta$ and  $\PP(a\lr b)>1-\delta$ for all $a\in A$, we have $\PP(0\lr b)>1-\eps$.
\end{conj}

Conjecture \ref{conj:epsdel} clearly follows from conjecture \ref{conj:postFKG} (by taking $\delta=1-\sqrt{1-\eps}$).

\begin{thm}\label{thm:conjtotheta}If conjecture \ref{conj:epsdel} holds then $\PP(|\Cl(0)|=\infty)=0$ at the critical probability for $\ZZ^d$ for any $d\ge 2$.
\end{thm}

The rest of this section is devoted to the proof of theorem \ref{thm:conjtotheta}. Since the case $d=2$ is well known \cite[\S 11]{Gri} we will assume throughout that $d\ge 3$ and is fixed.
Recall that $\theta(p)=\PP_p(|\Cl(0)|=\infty)$.

We start with two well-known facts.
\begin{lem}\label{lem:GM1}Assume $\theta(p)>0$. For every $r$ let $p_r$ be the probability that the box $[-r,r]^d$ intersects an infinite cluster. Then $\lim_{r\to\infty} p_r=1$.
\end{lem}
\begin{proof}By Kolmogorov's 0-1 law the probability that there exists an infinite cluster is 1. Since the events that $[-r,r]^d$ intersects an infinite cluster increase to the event that an infinite cluster exists, the result follows from the property of `continuity from below' of the probability measure.
\end{proof}
\begin{lem} \label{lem:Cerf}For every $m\ge 1$ and every $\eps>0$ there  exists an $M$ with the following property. Let $E$ be the event that there exist two different clusters in $[-M,M]^d$ both of which intersect both $[-m,m]^d$ and the internal vertex boundary of $[-M,M]^d$, $\piv [-M,M]^d$. Then $\PP(E)<\eps$.
\end{lem}
\begin{proof}See \cite[theorem 1.2]{Cerf} (since we do not need a quantitative relation between $M$ and $m$, this could also be concluded from \cite{AKN} or \cite{BK}).
\end{proof}

We now define some geometric notions. For $A\subseteq G$ we denote by $\piv A$ and $\pev A$ the internal and external vertex boundaries of $A$ respectively, that is
\[
\begin{aligned}
\piv A&:=\{a\in A: \exists v\in A^c,\: a\sim v\}\\
\pev A&:=\{v\in A^c: \exists a\in A,\: a\sim v\}.
\end{aligned}
\]
A box in $\ZZ^d$ is a subgraph of the form
 \[
 B=v+\prod[-s_{i},s_{i}]
 \]
for some $v\in\ZZ^d$ and some integers $s_1,\dotsc,s_d\geq 0$. Note that we accept a box with some of the sidelengths equal to zero. For such a box $B$ and for $R>0$ we denote by $\expand{B}{R}$ the enlargement of $B$ by $R$, that is,
 \[
 \expand{B}{R}=v+\prod[-s_{i}-R,s_{i}+R].
 \]

\begin{defn*}Let $F\subset Q\subset \RR^d$. We say that the couple $(F,Q)$ is a \emph{hittable geometry} if for every $p$ for which $\theta(p)>0$ and for every $\eps>0$ there exists an $m$ such that for any integer $\ell>m$ we have
  \[
  \PP_p([-m,m]^d\xleftrightarrow{lQ}lF)>1-\eps.
  \]
  Here $lQ=\{v\in\ZZ^d:v/l\in Q\}$, and similarly for $lF$.
\end{defn*}
Let us immediately note the symmetries of the definition.
\begin{lem}\label{lem:sym} If $(F,Q)$ is hittable geometry then so is $(TF,TQ)$ for $T$ either a coordinate permutation or the map $(x_1, x_2\dotsc,x_d)\mapsto(-x_1, x_2\dotsc,x_d)$. We call the group of transformation generated by these $T$ \emph{the lattice symmetries}.
\end{lem}
The proof is straightforward and we omit it.

The basic example of a hittable geometry is given by the following well-known lemma.
\begin{lem}\label{lem:quarterface}Let $Q=[-1,1]^d$ and let $F$ be a quarter face, namely the set $\{1\}\times[0,1]^{d-1}$. Then $(F,Q)$ is a hittable geometry.
\end{lem}
Of course, a quarter face is really a quarter of a face only in dimension 3, but we find the name sufficiently suggestive to use it in all dimensions.
\begin{proof} Using lemma \ref{lem:GM1} choose some $m$ such that
\[
\PP([-m,m]^d\lr\infty)>1-\eps^{d2^d}.
\]
Let now $\ell>m$ be some integer. Let $E$ be the event that $[-m,m]^d\lr \piv [-l,l]^d$, so that $\PP(E)>1-\eps^{d2^d}$. Let $F_1,\dotsc,F_{d2^d}$ be all the possible quarter faces (i.e.\ all maps of $F$ under lattice symmetries), and let $H_i$ be the event $[-m,m]^d\xleftrightarrow{[-\ell,\ell]^d}\ell F_i$. Then the lattice symmetries show that $\PP(H_i)$ does not depend on $i$. Further, since $E=\bigcup H_i$ and using FKG for the complement events we get
\[
1-\PP(E)\ge (1-\PP(H_i))^{d2^d}.\qedhere
\]
\end{proof}

A set will be called a `target' if from any point of a box it can be hit using a hittable geometry (i.e.\ it contains the corresponding $F$). Here is the precise definition.
\begin{defn*}
Let $B\subset D\subset \ZZ^d$ be two boxes, $\Hl$ be a finite collection of hittable geometries and $R\ge 0$ such that $\expand{B}{R}\subset D$. We say that a $T\subset D$ is a target with respect to  $(B,D,R,\Hl)$ if for every $v\in \expand{B}{R}$ there exists an integer $l(v)\ge R$ and a hittable geometry $(F(v),Q(v))\in \Hl$ such that
\[
v+l(v)Q(v)\subseteq D\qquad v+l(v)F(v)\subseteq T.
\]
\end{defn*}
For a typical example of a target $T$ (in $d=2$, for clarity), consider the hittable geometries being $\Hl=\{([-1,1]^2,F):F\textrm{ half face}\}$, take $R=0$, $B=[-s,s]^2$ for some $s$, $D=[-4s,4s]^{2}$ and take $T$ to be $\{2s\}\times[-3s,3s]$. It is easy to check that this $T$ satisfies the geometric condition. Do note, though, that different $v$ use different half faces --- all $v$ use a half of the face $\{1\}\times[1,1]$ but upper $v$ need the lower half face and vice versa. (We will give formal proof of similar claims below, e.g.\ during the proofs of lemmas \ref{lem:start} and \ref{lem:shrink}).


The graphs we will actually be using below are not $\ZZ^d$, but they do have large parts isomorphic to boxes in $\ZZ^d$. Hence the following definition will be useful.

\begin{defn*}
Let $G$ be a graph and $D$ a subset of vertices of $G$. We say that
$D$ is a subbox if the induced graph on $D$ is isomorphic to $\prod[-s_{i},s_{i}]$,
considered as a subgraph of $\mathbb{Z}^{d}$, for some $s_{i}$.
Throughout this section we will not distinguish in the notation between
$D$ and $\prod[-s_{i},s_{i}]$ e.g.\ we could write $(0,\dotsc,0)$
and refer to the vertex of $G$ given by applying the isomorphism
to $(0,\dotsc,0)$. Further, we require from a subbox that if any
edge that connects some vertex of $G\setminus D$ to some vertex $v\in D$
must have that $v\in\partial_{\textrm{iv}}\prod[-s_{i},s_{i}]$ (to be more precise, $v$ is an image, under the isomorphism just explained, of the internal vertex boundary in $\ZZ^d$).
\end{defn*}

The following lemma is the main lemma of the proof, and justifies all definitions (e.g.\ why did we call a target a target? Because the lemma roughly says that `if you can reach the center of the box, you can continue to the target, losing a little in the probability').

\begin{lem}
\label{lem:vincentless}Assume conjecture \ref{conj:epsdel} and let $p\in(0,1)$
satisfy that $\theta(p)>0$. Then for every $\eps>0$ there exists a $\delta>0$ such that for every finite set $\Hl$ of hittable geometries there
exist an $R$ which satisfies the following. If $G$ is a graph with subboxes
\[
B\subseteq \expand{B}{R}\subseteq D\subset G,
\]
and $T\subseteq D$ is a target with respect to $(B,D,R,\Hl)$ then for every vertex $o\in G\setminus D$ we have
\[
\mathbb{P}(o\lr B)>1-\delta \implies
\mathbb{P}(o\leftrightarrow T)>1-\varepsilon.
\]

\end{lem}

\begin{proof}[Proof, Step I] We first fix the parameters which will be used throughout the proof. The lemma is obvious for $\eps> 1$ so assume $\eps \le 1$. We use conjecture \ref{conj:epsdel} with $\eps_{\textrm{conjecture \ref{conj:epsdel}}}=\frac12\eps$. Our $\delta$ is related to the output of the conjecture by $\delta=\frac1{12}\eps \delta_{\textrm{conjecture \ref{conj:epsdel}}}$.

  With $\delta$ chosen and $\Hl$ given we can fix some geometric parameters. Let $\Hl'=\Hl\cup\{(\{1\}\times[-1,1]^{d-1},[-1,1]^d)\}$ i.e.\ a full face of $[-1,1]^d$ (this is a hittable geometry due to lemma \ref{lem:quarterface}). Using the definition of a hittable geometry choose $m\in\NN$ sufficiently large such that for all $\ell>m$ and all $(F,Q)\in\Hl'$,
\begin{equation}\label{eq:def m}
\mathbb{P}([-m,m]^{d}\xleftrightarrow{\ell Q}\ell F)>1-\tfrac{1}{d}\delta^2.
\end{equation}
Using lemma \ref{lem:Cerf} pick $M$ sufficiently large such that the probability that there
are two disjoint clusters from $[-m,m]^{d}$ to $\piv[-M,M]^{d}$
is smaller than $\delta^{2}$. 
Let $k$ be sufficiently large so that $(1-p^{(d+1)(4M)^{d-1}})^{k}<\delta$.
Let $k_{2}$ be sufficiently large such that $(1-(1-p)^{k(4M)^{d-1}})^{k_{2}}<\frac{1}{2}\delta$.
Let $R=\lceil\frac{2}{\delta}k_{2}\rceil+5M$.\\

\textit{Step II}. Let $B\subseteq \expand{B}{R}\subseteq D$ be given and assume that $o\in G\setminus D$. 
%
For every
$j\in[5M,R]$ let $E_{j}$ be the event that
\[
\Big|\mathcal{C}\Big(o;G\setminus \expand{B}{j}\Big)\cap\partial_{\textrm{ev}}\expand{B}{j}\Big|\ge k(4M)^{d-1}.
\]
We claim that (under the assumption $\PP(o\leftrightarrow B)>1-\delta$ from the statement of the lemma) there exists some $j$ which satisfies
$\mathbb{P}(E_{j})>1-2\delta$.

To see this assume that the contrary
is true, i.e.\ $\mathbb{P}(E_{j})\le 1-2\delta$ for
all $j$. Consider the variable
\[
X:=\sum_{j=5M}^{R}\mathbbm{1}\Big\{\Big\{ o\leftrightarrow B\Big\}\setminus E_{j}\Big\}.
\]
On the one hand our assumption that $\PP(o\lr B)>1-\delta$ combined with  $\mathbb{P}(E_{j})\le 1-2\delta$ for all $j$
implies that $\mathbb{E}(X)\ge\delta(R-5M)$. On the other hand we claim
that
\begin{equation}
\mathbb{P}(X>k_{2})<\tfrac{1}{2}\delta.\label{eq:X>k2}
\end{equation}
Essentially, \eqref{eq:X>k2} follows because every time $E_j$ did not occur we have probability at least $(1-p)^{k(4M)^{d-1}}$ to disconnect from $\expand{B}{j}$. To make this a little more formal, define $j_0=R+1$ and for any $i>0$,
\[
j_i\coloneqq \max\{\max\{j<j_{i-1}:\neg E_j\},5M-1\}.
\]
and let $G_i$ be the events
\[
G_i\coloneqq \{j_i> 5M-1\}\cap \{0\leftrightarrow \expand{B}{j_i}\}.
\]
We claim that
\begin{equation}\label{eq:Gicond}
\PP(G_i\,|\,G_1,G_2,\dotsc,G_{i-1})\le 1-(1-p)^{k(4M)^{d-1}}.
\end{equation}
Indeed, if $j_{i}=5M-1$ then this probability is simply zero, while in the other case, because $E_{j_{i}}$ did not happen we have that the cluster of $o$ contains at most $4(kM)^{d-1}$ vertices in $\pev \expand{B}{j_i}$. Each such vertex has exactly one edge connecting it to $\expand{B}{j_i}$ and hence the probability that they are all closed is at least $(1-p)^{k(4M)^{d-1}}$. This shows \eqref{eq:Gicond}. We conclude that
\[
\PP(G_1\cap\dotsb\cap G_{k_2})\le \big(1-(1-p)^{k(4M)^{d-1}}\big)^{k_2}<\tfrac12\delta
\]
by the definition of $k_2$. But $X>k_2$ implies $G_1\cap\dotsb\cap G_{k_2}$. Thus \eqref{eq:X>k2} follows.

 We conclude
that
\[
\mathbb{E}(X)<k_{2}+\tfrac{1}{2}\delta(R-5M)\le\delta(R-5M)
\]
by the definition of $R$. This contradicts our assumption and hence we get
that for some $j$, $\mathbb{P}(E_{j})>1-2\delta$.
Fix one such $j$ until the end of the proof.\\

\textit{Step III}.
For the value of $j$ defined at the end of Step II, write
$\partial_{\textrm{iv}}\expandp{B}{j}$ as a disjoint union of \emph{plaquettes}. We define a plaquette as a subset of $\piv\expandp{B}{j}$ of the form
\[
w+[0,l_1-1]\times\dotsb\times\{0\}\times\dotsb\times[0,l_d-1],
\]
for some $w\in\ZZ^d$, some position for the $\{0\}$ and some $l_i$ (all of which might depend on the plaquette). We require $l_{i}\in[2M+2,4M-1]$. For example,
\[
\{s_1+j\}\times[s_2+j-2M-1,s_2+j]\times\dotsb\times[s_d+j-2M-1,s_d+j]
\]
is a legal plaquette.
It is easy to check that $\piv\expandp{B}{j}$ can be written as a disjoint union of plaquettes (note that we use here that $j\ge 5M$, and that $M$ is sufficiently large) and we will not show it in detail. Let $\Pl$ be such a set of plaquettes.

Call a given plaquette $P\in\Pl$ a seed if all edges between any two of its
vertices are open. Further, require that all edges from $P$
into $\expand{B}{j-1}$ be open. Let $\mathcal{G}$ be
the event that $E_{j}$ occurred and in addition $\mathcal{C}(o;G\setminus \expand{B}{j-1})$
contains at least one seed. We now claim that

\begin{equation}
\mathbb{P}(\Gl)>1-3\delta.\label{eq:G lem 1}
\end{equation}
This is because if $E(j)$ occurred then $\mathcal{C}(o;G\setminus \expand{B}{j})$
contains at least $k(4M)^{d-1}$ points in $\partial_{\textrm{ev}} \expandp{B}{j}$
and hence faces at least $k$ different plaquettes in $\Pl$. 
Denote these plaquettes by $P_{1},\dotsc,P_{k}$.
For each plaquette $P_i$ the probability that it is a seed and connects
to $o$ in $G\setminus \expandp{B}{j-1}$ is at least $p^{(d+1)(4M)^{d-1}}$ as there are no
more than $(d-1)(4M)^{d-1}$ internal edges to the plaquette, no more
than $(4M)^{d-1}$ connecting the plaquette to $\expand{B}{j-1}$,
and one more edge is required to connect the plaquette to $o$. This estimate holds also after conditioning on $\Cl(o;G\setminus\expandp{B}{j})$. For any two plaquettes
the events that they are seeds are independent, again also after conditioning on
$\mathcal{C}(o;G\setminus \expandp{B}{j})$. We get
\[
\mathbb{P}(E_{j}\setminus\mathcal{G})\le\big(1-p^{(d+1)(4M)^{d-1}}\big)^{k}<\delta,
\]
where the second inequality is by the definition of $k$. Together
with the definition of $j$ we get (\ref{eq:G lem 1}).

Order the plaquettes in $\Pl$ by some arbitrary deterministic order, and for
each $P$ let $\mathcal{F}_{P}$ be the event that $E_j$ occurred, $P$ is a seed in the cluster of $0$,
and further, $P$ is the first such seed in this order. These events are disjoint,
so
\begin{equation}
\sum_{P}\mathbb{P}(\mathcal{F}_{P})=\mathbb{P}(\mathcal{G})>1-3\delta.\label{eq:sumPFP}
\end{equation}

\textit{Step IV}. Assume that $T\subseteq D$ is a target with respect to $(B,D,R,\Hl)$.
Let
\[
S=\expandp{B}{j-1}\setminus \expandp{B}{j-2-2M}.
\]
Examine a $v\in S$ such that $v+[-M,M]^d\subseteq S$. Since $S$ has
thickness $2M$, this means that $v+[-M,M]^d$ has at least one face
facing `outside' i.e.\ contained in $\partial_{\textrm{iv}}\expandp{B}{j-1}$.
Let $\mathcal{U}(v)$ denote the set of such faces.

Let $\mathcal{G}(v)$ be the event that there exists a cluster
$\mathcal{C}$ with the following properties
\begin{enumerate}
\item $\Cl$ intersects both $v+[-m,m]^{d}$ and $v+\piv[-M,M]^{d}$ and is unique with respect to this property.
\item For all $U\in\mathcal{U}(v)$ we have that $\Cl\cap (v+[-M,M]^d)$ contains a path from $v+[-m,m]^d$ to $U$.
\item $\mathcal{C}\cap T\ne\emptyset$.
\end{enumerate}
We claim that
\begin{equation}\label{eq;v very connected}
\mathbb{P}(\mathcal{G}(v))>1-3\delta^{2}.
\end{equation}
Indeed, the fact that $T$ is target with respect to $(B,D,R,\Hl)$ (and the fact that $v\in \expand{B}{R}$) implies that there exists $\ell(v)>R$ and a hittable geometry $(F(v),Q(v))\in\Hl$ such that $v+\ell(v) F(v)\subseteq T$ and $v+\ell(v) Q(v)\subseteq D$. The definition of $m$ in \eqref{eq:def m} now gives that
\begin{equation}
  \mathbb{P}(v+[-m,m]^{d}\leftrightarrow T)
  \ge\mathbb{P}(v+[-m,m]^{d}\leftrightarrow v+\ell(v)F(v))
  >1-\tfrac{1}{d}\delta^{2}.\label{eq:vtoF}
\end{equation}
Further, as $m$ was chosen using not only the hittable geometries $\Hl$, but also that of the full face, we may use lattice symmetries to get 
\[
\mathbb{P}(v+[-m,m]^{d}\xleftrightarrow{v+[-M,M]^d} U)>1-\tfrac{1}{d}\delta^{2},
\]
for $U\in\mathcal{U}(v)$. Summing over all $U\in\mathcal{U}(v)$ gives
\begin{equation}
\mathbb{P}(v+[-m,m]^{d}\xleftrightarrow{v+[-M,M]^d} U\;\forall U\in\mathcal{U}(v))>1-\delta^{2}.\label{eq:vtoU}
\end{equation}
The inequality (\ref{eq;v very connected}) now follows from (\ref{eq:vtoF}), (\ref{eq:vtoU}) and the definition of $M$.

For $\xi\subseteq E(S)$ where $E(S)$ is the set of all edges which
have both ends in $S$, consider the event that all edges in $\xi$ are open and all other edges in $E(S)$ are closed. For brevity we denote this event by $\xi$ and let $p_{\xi}$ denote the
probability that it occurred,
that is $p_{\xi}=p^{|\xi|}(1-p)^{|E(S)\setminus\xi|}$). With this notation we can write  (\ref{eq;v very connected}) as
\[
\sum_{\xi\subseteq E(S)}p_{\xi}\mathbb{P}(\mathcal{G}(v)\,|\,\xi)>1-3\delta^{2}.
\]
This implies
\[
\sum_{\xi\subseteq E(S)}p_{\xi}\mathbbm{1}\{\mathbb{P}(\mathcal{G}(v)\,|\,\xi)>1-\delta\}>1-3\delta.
\]
Denote by $\tilde{A}_{\xi}$ the collection of all clusters $\mathcal{C}$ in $E(S)$ which satisfy $\mathbb{P}(\mathcal{C}\lr T\,|\,\xi)>1-\delta$. Applying the uniqueness condition in $\mathcal{G}(v)$ we find that the last inequality implies
\begin{equation}\label{eq: normaly in A before translation}
\sum_{\xi\subseteq E(S)}p_{\xi}\mathbbm{1}\{\exists \mathcal{C}\in \tilde{A}_{\xi},\mathcal{C}\cap U\ne\emptyset\;\forall U\in\mathcal{U}(v)\}>1-3\delta.
\end{equation}

\textit{Step V}. For a given $\xi\subseteq E(S)$ define an auxiliary graph $K_{\xi}$ by
taking the graph $G$, removing every edge in $E(S)\setminus\xi$,
and shortening every edge in $\xi$ (i.e.\ the new vertices are
equivalence classes of vertices in $S$, where two vertices $x$ and
$y$ are considered equivalent if there is a path $\gamma=\{x=x_{1},x_{2},\dotsc,x_{\ell}=y\}$
with $(x_{i},x_{i+1})\in\xi$ for all $i=1,\dotsc,\ell-1$). Vertices
and edges outside $S$ are unaffected. For brevity, we identify any subset of $S$ with the
subset of the vertices of $K_{\xi}$ which are the equivalence classes
of all elements of it. Let $A_{\xi}\subseteq K_{\xi}$
be the set of all vertices $w\in S$ which satisfy
\[
\mathbb{P}_{K_{\xi}}(w\leftrightarrow T)>1-\delta,
\]
where $\mathbb{P}_{K_{\xi}}$ denotes percolation on $K_{\xi}$ induced by this process (i.e.\ $A_{\xi}$ are the equivalence classes obtained from the set $\tilde{A}_{\xi}$, defined at the end of Step IV). For clarity, from now on we will also denote probabilities
in $G$ by $\mathbb{P}_{G}$. Let
\[
\phi_{\xi}:=\mathbb{P}_{K_{\xi}}(o\leftrightarrow A_{\xi}).
\]
Our plan is to estimate $\phi_\xi$ and then use conjecture \ref{conj:epsdel}. 

To start, we wish to translate (\ref{eq: normaly in A before translation}) to the language of $K_\xi$. We get
\begin{equation}
\sum_{\xi}p_{\xi}\mathbbm{1}\Big\{\exists w\in A_{\xi}\cap\bigcap_{U\in\mathcal{U}(v)}U\Big\}>1-3\delta\label{eq:normally in A}
\end{equation}
(here $U$ is a subset of vertices of $K_\xi$ and $\Ul$ is a family of such subsets).
Next, we recall the plaquettes and seeds defined in Step III as well as the events $\mathcal{F}_{P}$. An obvious, but crucial, point is that these events
can be defined in each $K_{\xi}$ and their probabilities do not depend
on $\xi$ i.e.\ $\mathbb{P}_{K_{\xi}}(\mathcal{F}_{P})=\mathbb{P}_{G}(\mathcal{F}_{P})$. For a plaquette $P$ let $v(P)\in S$ be a point such that some
$U\in\mathcal{U}(v(P))$
neighbours $P$ (such a point always exists, even for a plaquette touching the $(d-2)$-skeleton of the cube, because a plaquette of side length $2M+2$ still has $2M+1$ of its length facing $S$). Denote the $U\in\mathcal{U}(v(P))$ just chosen by
$U(P)$ for short. Then
\begin{equation}
  \mathbb{P}_{K_{\xi}}(o\leftrightarrow A_{\xi})
  \ge\sum_{P\in\Pl}\mathbb{P}_{K_{\xi}}(\mathcal{F}_{P})\mathbbm{1}\{\exists w\in A_{\xi}\cap U(P)\}\label{eq:but crucial}
\end{equation}
since if for some $P$ the event $\mathcal{F}_{P}$ occurred and such
a $w$ exists then $P$ connects to $w$ (as the event that $P$
is a seed implies that all edges from $P$ to $U(P)$ are open) and
hence $o$ connects to $A_{\xi}$. We conclude
\begin{align*}
\sum_{\xi}p_{\xi}\phi_{\xi} & =\sum_{\xi}p_{\xi}\mathbb{P}_{K_{\xi}}(o\leftrightarrow A_{\xi})\\
 & \stackrel{\mathclap{\textrm{(\ref{eq:but crucial})}}}{\ge}\sum_{\xi}p_{\xi}\sum_{P\in\Pl}\mathbb{P}_{K_{\xi}}(\mathcal{F}_{P})\mathbbm{1}\{\exists w\in A_{\xi}\cap U(P)\}\\
 & =\sum_{P\in\Pl}\mathbb{P}_{G}(\mathcal{F}_{P})\sum_{\xi}p_{\xi}\mathbbm{1}\{\exists w\in A_{\xi}\cap U(P)\}\\
& \stackrel{\mathclap{\textrm{(\ref{eq:normally in A})}}}{\ge}
(1-3\delta)\sum_{P\in\Pl}\mathbb{P}_{G}(\mathcal{F}_{P})
\stackrel{\textrm{(\ref{eq:sumPFP})}}{>}(1-3\delta)^2>1-6\delta.
\end{align*}
Let $\Wl=\{\xi:\phi_\xi> 1-12\delta/\eps\}$. The above calculation shows that $\sum_{\xi\notin\Wl}p_\xi\le \frac12\eps$. For every $\xi\in\Wl$ we can use conjecture \ref{conj:epsdel} (recall our definition $\delta=\frac1{12}\eps \delta_{\textrm{conjecture \ref{conj:epsdel}}}$ in the very beginning of the proof of the lemma) for $K_\xi$ and get
\[
\PP_{K_\xi}(o\lr T)\ge 1-\tfrac 12\eps.
\]
(We identified the set $T$ to a point and used that point as the $b$ of conjecture \ref{conj:epsdel}). Summing over $\xi$ and using the above estimate for $\xi\notin\Wl$ gives
\[
\sum p_\xi\PP_{K_\xi}(o\lr T)\ge (1-\tfrac 12\eps)^2>1-\eps.
\]
But this means that $\PP_G(o\lr T)>1-\eps$, proving the lemma.
\end{proof}

\begin{lem}
\label{lem:start}Assume $\theta(p)>0$ and conjecture \ref{conj:epsdel}. Then for every $\varepsilon>0$ and every integer $K\ge 2$ there exists an
$m>0$ such that for any integer $r\ge m$ we
have
\[
\mathbb{P}([-m,m]^{d}\xleftrightarrow{[-r,Kr]\times[-r,r]^{d-1}}\{Kr\}\times[-r,r]^{d-1})>1-\varepsilon.
\]
\end{lem}
That is, if $K\ge 2$ is an integer then $(\{K\}\times[-1,1]^{d-1},[-1,K]\times[-1,1]^{d-1})$ is a hittable geometry.

\begin{proof}
  We start by defining $\eps_0\coloneqq\eps,\eps_1,\dotsc,\eps_{4K}$ and $R_1,\dotsc,R_{4K}$ using lemma \ref{lem:vincentless} inductively.  All applications of lemma \ref{lem:vincentless} are with the hittable geometry of all quarter faces (recall lemma \ref{lem:quarterface}), which we will denote by $\Hl$. As for the $\eps$ parameters of lemma \ref{lem:quarterface}, for every $i\in\{1,\dotsc,4K\}$ we apply lemma \ref{lem:vincentless} with its input being $\eps_{\textrm{lemma \ref{lem:vincentless}}}=\eps_{i-1}$ and use its output to define $\eps_{i}=\min\{\eps_{i-1},\delta_{\textrm{lemma \ref{lem:vincentless}}}\}$ and $R_{i}=R_{\textrm{lemma \ref{lem:vincentless}}}$. Define $R\coloneqq\max \{R_i\}$, $\delta\coloneqq\min\{\eps_i\}=\eps_{4K}$.

  Next, using lemma \ref{lem:quarterface} choose an $n$ so large such that for any $r>n$ we have
\begin{equation}
\mathbb{P}([-n,n]^{d}\xleftrightarrow{[-\frac{1}{2}r,\frac{1}{2}r]^{d}}\{\tfrac{1}{2}r\}\times[-\tfrac{1}{2}r,\tfrac{1}{2}r]^{d-1})>1-\delta.\label{eq:n one face}
\end{equation}
We define $m=\max\{10n,16KR\}$, and let $r\ge m$ be arbitrary.

We now
apply lemma \ref{lem:vincentless}, with the hittable geometry of the quarter faces $\Hl$, to the graph $\Omega$ which one
gets by taking $[-r,Kr]\times[-r,r]^{d-1}$ and identifying the cube $[-n,n]^{d}$
to a point (which will be $o$). We take the subbox $D$, the intermediate stage $B$ and the target $T$ to be
\begin{align*}
  D&=[\tfrac{1}{8}r,\tfrac{3}{4}r]\times[-r,r]^{d-1},\\
  B&=\{\tfrac12r\}\times[-\tfrac12r,\tfrac12r]^{d-1}\\
  T&=\{\tfrac{3}{4}r\}\times[-\tfrac{1}{2}r-R,\tfrac{1}{2}r+R]^{d-1}.
\end{align*}
It is easy to see that $T$ is a target with respect to $(B,D,R,\Hl)$. Indeed, we need to show that for any $v\in [\frac12r-R,\frac12r+R]\times[-\tfrac12r-R,\tfrac12r+R]^{d-1}$ there exists an $\ell(v)$ and a quarter face $F(v)$ so that $v+[-\ell(v),\ell(v)]^d\subseteq D$ and $v+\ell(v)F(v)\subseteq T$. To see this choose $\ell(v)=\tfrac34r-v_1$ ($v_1$ is simply the first coordiante of $v$). Since $|v_1-\tfrac12r|\le R$ we get $\ell(v)\le \frac14r+R$ and then
\[
v+[-\ell(v),\ell(v)]^d\subseteq [\tfrac14r-2R,\tfrac34r]\times[-\tfrac34r-2R,\tfrac34r+2R]^{d-1}\subset D.
\]
Next define
\[
F(v)=\prod_{i=1}^{d}\begin{cases}
\{1\} & i=1\\{}
[0,1] & i>1,\;v_{i}\le0\\{}
[-1,0] & i>1,\;v_{i}>0.
\end{cases}
\]
It is straightforward to check that $v+\ell(v)F(v)\subseteq T$.

Thus we may apply lemma \ref{lem:vincentless} with $\eps_{\textrm{lemma \ref{lem:vincentless}}}=\eps_{4K}$. Since $\PP(o\lr B)>1-\eps_{4K}$ and since $\ell(v)\ge R_{4K}$ for all $v$ we get
\[
\mathbb{P}_\Omega(o\lr T)>1-\eps_{4K-1}.
\]
We repeat this $4K-3$ more times, each time moving
the face by $\frac{1}{4}r$, inflating it by $R$ and moving from $\eps_{i+1}$ to $\eps_i$ and get
\begin{multline*}
\mathbb{P}_\Omega(o\lr\{Kr\}\times[-r,r]^{d-1})\\
\ge\mathbb{P}_\Omega(o\lr\{Kr\}\times[-\tfrac{1}{2}r-(4K-2)R,\tfrac{1}{2}r+(4K-2)R]^{d-1})>1-\eps_2\ge 1-\varepsilon.
\end{multline*}

Going back to $\mathbb{Z}^{d}$ we get
\[
\mathbb{P}([-n,n]^{d}\xleftrightarrow{[-r,Kr]\times[-r,r]^{d-1}}\{Kr\}\times[-r,r]^{d-1})>1-\varepsilon.
\]
Since $n\le m$, the lemma is proved.
\end{proof}
\begin{lem}\label{lem:shrink}Let $p$ satisfy $\theta(p)>0$ and assume conjecture \ref{conj:epsdel}. Then for every $\eps>0$ there exists a $\delta>0$ and an $m$ such that for every $r>m$ and every graph $G$ containing a subbox $D=[-5r,25r]\times [-5r,5r]^{d-1}$ and every $o\notin D$,
  \begin{multline*}
  \PP(o\xleftrightarrow{A} [-3r,3r]^d)>1-\delta\implies
  \PP(o\xleftrightarrow{U} (20r,0,\dotsc,0)+ [-3r,3r]^{d})>1-\eps,\\
  A\coloneqq G\setminus[5r,25r]\times[-5r,5r]^{d-1}\qquad U\coloneqq A\cup [5r,22r]\times[-2r,2r]^{d-1}
  \end{multline*}
\end{lem}
\begin{figure}
  \centering
  \input{ABD.pstex_t} 
  \caption{The geometry of lemma \ref{lem:shrink}.}\label{cap:shrink}
\end{figure}
See figure \ref{cap:shrink}.
\begin{proof}
  We first apply lemma \ref{lem:vincentless} with the hittable geometry
  \[
  (\{88\}\times[-1,1]^{d-1},[-1,88]\times[-1,1]^{d-1})
  \](in other words, with the hittable geometry from lemma \ref{lem:start} with $K=88$), and with $\eps_{\textrm{lemma \ref{lem:vincentless}}}=\eps$. Use the output of the lemma to define $\eps_2\coloneqq \delta_{\textrm{lemma \ref{lem:vincentless}}}$ and $R_2\coloneqq R_{\textrm{lemma \ref{lem:vincentless}}}$.

  Next we use lemma \ref{lem:vincentless} $d$ more times inductively to define $\eps_3,\dotsc,\eps_{d+2}$ and $R_3,\dotsc,R_{d+2}$. This time we use it just with the quarter face hittable geometry (lemma \ref{lem:quarterface}). That is, for every $i\in\{3,\dotsc,d+2\}$ we apply lemma \ref{lem:vincentless} with its input being $\eps_{\textrm{lemma \ref{lem:vincentless}}}=\eps_{i-1}$ and use its output to define $\eps_i\coloneqq\min\{\eps_{i-1},\delta_{\textrm{lemma \ref{lem:vincentless}}}\}$ and $R_i\coloneqq R_{\textrm{lemma \ref{lem:vincentless}}}$. Define $m\coloneqq \max\{8d\max\{R_i\},88d\}$ and  $\delta\coloneqq\min\{\eps_i\}=\eps_{d+2}$. Let $r>m$ be arbitrary.

  We now define the geometric parameters for the $d+1$ applications of lemma \ref{lem:vincentless}. The first $d$ applications are with the hittable geometry of the quarter face $\Hl$. We start by applying it in the subgraph $A$ with $R=m/8d$ and with
  \begin{gather*}
    D= [-5r,5r]^d\qquad  B= [-3r,3r]^d\\
    T= [-\tfrac32r-R,\tfrac32r+R]\times[-3r-R,3r+R]^{d-1}.
  \end{gather*}
  Let us check that $T$ is indeed a target with respect to $(B, D, R, \Hl)$. Indeed, if $(v_1,\dotsc,v_d)\in [-3r-R,3r+R]^d$ then we define $\ell(v)= \frac32r$, $Q(v)=[-1,1]^d$ and
  \[
  F(v)=\prod_{i=1}^d \begin{cases}
  \{-\sign(v_1)\} & i=1\\
  [0,1] & i>1,v_i<0\\
  [-1,0] & i>1,v_i\ge 0.
  \end{cases}
  \]
  It is straightforward to check that $v+\ell(v)F(v)\subset T$ and $v+\ell(v)Q(v)\subset D$ for all $v$ (note that we used here $R<r/2$). We get that $\PP(o\xleftrightarrow{A} T)>1-\eps_{d+1}$.

  The second application of lemma \ref{lem:vincentless} halves the second dimension. Namely we use it, still in the subgraph $A$ and with $R=m/8d$, but this time our geometric parameters are
  \begin{gather*}
    D= [-5r,5r]^d\qquad  B= [-\tfrac32r-R,\tfrac32r+R]\times[-3r-R,3r+R]^{d-1}\\
    T= [-\tfrac32r-2R,\tfrac32r+2R]^2\times[-3r-2R,3r+2R]^{d-2}.
  \end{gather*}
  To check that $T$ is a target, pick an arbitrary $(v_1,\dotsc,v_d)\in [-\frac32r-2R,\frac32r+2R]\times[-3r-2R,3r+2R]^{d-1}$ and define $\ell(v)=\frac32r$ and
  \[
  F(v)=\prod_{i=1}^d \begin{cases}
  \{-\sign(v_2)\} & i=2\\
  [0,1] & i\ne 2,v_i<0\\
  [-1,0] & i\ne 2,v_i\ge 0.
  \end{cases}
  \]
  Again, the verification is straightforward.

  We continue this way, each time halving one dimension while inflating all the rest by $R$, and moving from $\eps_{i+1}$ to $\eps_i$ until we finally end up with
  \[
  \PP(o\xleftrightarrow{A} [-\tfrac32r-dR,\tfrac32r+dR]^d)>1-\eps_2.
  \]
  The final application of lemma \ref{lem:vincentless} is using the geometry we started the lemma with. Here are the exact details. We use the subgraph $U$, $R=m/8d$ and
  \begin{gather*}
    D=[-2r,22r]\times[-2r,2r]^{d-1}\qquad B=[-\tfrac32r-dR,\tfrac32r+dR]^d\\
    T=(20r,0,\dotsc,0)+[-2r,2r]^{d}
  \end{gather*}
  For every $v\in[-\tfrac32r-(d+1)R,\tfrac32r+(d+1)R]^d$ we choose $\ell(v)=\lfloor(20r-v_1)/88\rfloor$, $F(v)=\{88\}\times[-1,1]^{d-1}$ and $Q(v)=[-1,88]\times[-1,1]^{d-1}$. This definition implies that $\ell(v)< \frac14r$ and then the two requirements, $v+\ell(v)Q(v)\subseteq D$ and $v+\ell(v)F(v)\subseteq T$, are both straightforward to verify. Lemma \ref{lem:vincentless} then gives us that
  \[
  \PP(o\xleftrightarrow{U} T)>1-\eps,
  \]
  proving the lemma.
\end{proof}
The following definition encompasses the well-known concept of an exploration process.
\begin{defn*}
  A (supercritical, two-dimensional) exploration process is a couple of sequences of random subsets $G_i\subseteq X_i\subset\ZZ^2$ with the following properties:
  \begin{enumerate}
  \item $X_1=\{0\}$.
  \item $G_i\subseteq G_j$ and $X_i\subseteq X_j$ for all $j>i$.
  \item If there exists a $v\in\ZZ^2\setminus X_i$ neighbouring $G_i$ then the first of the $v$ satisfying this property (in lexicographic order) has $X_{i+1}=X_{i}\cup\{v\}$ and $G_{i+1}\subseteq G_i\cup\{v\}$. If no such $v$ exists then $X_{i+1}=X_i$ and $G_{i+1}=G_i$.
  \item There exists some $p>p_c(\ZZ^2,\textrm{site})$, i.e.\ $p$ is bigger than the critical threshold for vertex percolation on $\ZZ^2$, with the following property. Whenever there exists a $v$ as in (3) then
    \begin{equation}\label{eq:good>p}
    \PP(v\in G_{i+1}\,|\,G_1,\dotsc,G_i,X_1,\dotsc,X_i)>p.
    \end{equation}
    \item $\PP(G_1=\{0\})>0$.
  \end{enumerate}
\end{defn*}
It is well-known and easy to see that if $(G_i,X_i)$ is an exploration process then
\[
\PP(\lim_{i\to\infty}|G_i|=\infty)>0.
\]
We skip the details, but remark that \cite[lemma 1]{GM} is similar.

\begin{rem*}Here and below, there is nothing special in the lexicographic order. Any deterministic, predefined order on $\ZZ^2$ would have worked equally well.
\end{rem*}
\begin{proof}[Proof of theorem \ref{thm:conjtotheta}]
Let $p$ be such that $\theta(p)>0$. We will show that at $p$ there
is percolation in some slab, which will imply, by Aizenman-Grimmett \cite[\S\S 3.2-3.3]{Gri},
that $p>p_{c}$. Since $p$ was arbitrary with $\theta(p)>0$, this
will show the claim.

Let $\varepsilon$ satisfy $\varepsilon<1-p_{c}(\mathbb{Z}^{2},\text{site})$.
Let $\delta>0$ and $m$ be given by lemma \ref{lem:shrink} used with $\eps_{\textrm{lemma \ref{lem:shrink}}}=\frac18\eps$. Let $m_2$ be given by lemma \ref{lem:start} with $\eps_{\textrm{lemma \ref{lem:start}}}=\delta$ and $K=3$.

We now use a peculiar feature of lemma \ref{lem:vincentless} that we have not used before: that $\delta$ does not depend on the hittable geometry being used, only $R$ does. We use lemma \ref{lem:vincentless} with $\eps_{\textrm{lemma \ref{lem:vincentless}}}=\delta$ and call the first output $\delta_2$, namely $\delta_2\coloneqq\delta_{\textrm{lemma \ref{lem:vincentless}}}$. We next choose an integer $K\ge 20$, so large that $(1-\delta_2)^K<\frac18\eps$. We use lemma \ref{lem:start} to show that $(\{2K\}\times[-1,1]^{d-1},[-1,2K]\times[-1,1]^{d-1})$ is a hittable geometry, and then return to lemma \ref{lem:vincentless} with this hittable geometry to get $R$. We let $r$ be some number with $r\ge 2K\max\{m,m_2,R\}$ and divisible by $K$. This finishes the choice of parameters. 


To continue, we need some geometric notation.
For each $v\in\mathbb{Z}^{2}$ we associate two concentric cubes $M_v\subset Q_{v}\subset\mathbb{Z}^{d}$
by
\[
M_v\coloneqq 20vr+[-3r,3r]^d\qquad Q_{v}\coloneqq 20vr+[-5r,5r]^{d}
\]
where the $v$ on the right hand sides are in fact $(v_1,v_2,0,\dotsc,0)\in\ZZ^d$. 

For an $x$ neighbouring $v$ define the
`edge' $E_{v,x}$ to be the $(10r)^d$
box between them, as well as $Q_x$. For example
\[
E_{v,v+(1,0)}=20vr+[5r,25r]\times[-5r,5r]^{d-1}.
\]
and similarly the other 3 neighbours of $v$. We will also need a narrowed version which we will denote by $H_{v,x}$. Again we take the right neighbour as an example:
\[
H_{v,v+(1,0)}=20vr+[5r,22r]\times[-2r,2r]^{d-1}.
\]
For every $j\in\{0,\dotsc,K\}$ we define
the partial edge $E^j_{v,x}$ to be the $j/K$ part of the first half of $E_{v,x}$ touching $Q_{v}$, and we let $H_{v,x}^{j}$ be its narrowed version. For example
\begin{align*}
  E_{v,v+(1,0)}^{j}&=20vr+[5r,5r+10rj/K]\times[-5r,5r]^{d-1}.\\
  H_{v,v+(1,0)}^{j}&=20vr+[5r,5r+10rj/K]\times[-2r,2r]^{d-1}.
\end{align*}
See figure \ref{cap:Eisimp}.

\begin{figure}
  \input{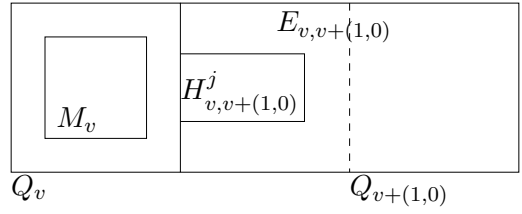}
  \caption{The set $Q_v$, $M_v$, $E$ and $H$.}\label{cap:Eisimp}
\end{figure}

The proof revolves around an exploration process. We will construct a random $(\El_i,G_i,X_i)_{i=1}^\infty$ where $\El_i\subset\ZZ^2\times[-5r,5r]^{d-2}$, $\El_i\subseteq \El_{i+1}$ and $G_i\subseteq X_i\subset \ZZ^2$ satisfying the following properties
\begin{enumerate}
\item $G_i$ and $X_i$ are functions of $\omega|_{\El_i}$. Further,
  \begin{equation}\label{eq:Ei and Xi}
  v\in X_i\iff Q_v\subseteq \El_i\iff Q_v\cap\El_i\ne\emptyset.
  \end{equation}
\item \label{pg:GiXi}$(G_i,X_i)$ is a two-dimensional, supercritical exploration process.
\item For every $v=(v_1,v_2)\in G_i$, $0\stackrel{\El_i}{\lr}(v_1,v_2,0,...,0)$.
\end{enumerate}
We remark that (3) would not be an if and only if --- there might be $v\in X_i\setminus G_i$ which are also connected to 0 in $\El_i$. This will not matter for us. What is clear at this point is that the combination of (2) and (3) proves that the probability that there is percolation in $\ZZ^2\times[-5r,5r]^{d-2}$ is positive, and hence $p>p_c$, as promised. So we only need to show that $\El_i,G_i,X_i$ as above indeed exist. We will first explain how to construct the process (with special note on the fact that it would satisfy formula \eqref{eq:goodedge 2.5} below) and then prove its properties.

We start by declaring that $G_0=\{(0,0)\}$ if all edges of $Q_{(0,0)}$ are open (and declare $\El_1=Q_{(0,0)}$).

Assume now that $\El_i,G_i,X_i$ have been constructed. If there is no $v\in\ZZ^2\setminus X_i$ neighbouring some $w\in G_i$ then there is nothing to do, so assume this is not the case, and let $v$ be the minimal in the lexicographic order satisfying this property. Let $w\in G_i$ be the neighbour of $v$ (if more than one exists, take the first in the lexicographic order). Define $X_{i+1}=X_i\cup\{v\}$.

Denote by $\mathcal{X}=\mathcal{X}_v$ the set of neighbours of $v$ which are not
in $X_{i}$ (in particular $w\not\in\mathcal{X}$). For
every $x\in\mathcal{X}$ we will declare whether the connection $v$-$x$
is good or not separately (these events will not be independent between
the different $x$), and declare $v$ good (i.e.\ define $G_{i+1}=G_i\cup\{v\}$) if all connections to all
$x\in\mathcal{X}$ are good. Fix therefore one $x\in\mathcal{X}$.
We declare the connection $v$-$x$ good if there exists a choice
of $j\in\{0,\dotsc,K-1\}$ such that
\begin{gather}
\mathbb{P}(0\xleftrightarrow{\mathcal{D}\cup E_{v,x}}M_{x}\,|\,\omega|_{\mathcal{D}})>1-\delta\label{eq:goodedge 2.5}\\
\mathcal{D}\coloneqq\mathcal{D}_{i,x}=\mathcal{E}_{i}\cup E_{w,v}\cup H_{v,x}^{j}\nonumber
\end{gather}
If such a $j$ may be found let $j_{x}$ be the first one, and if
not let $j_{x}=K-1$, and declare the connection $v$-$x$ bad. In
both cases ($v$ is good, i.e.\ all connections are good, or $v$
is bad i.e.\ at least one connection is bad) declare
\[
\mathcal{E}_{i+1}=\mathcal{E}_{i}\cup E_{w,v}\cup\bigcup_{x\in\mathcal{X}}H_{x}^{j_{x}}.
\]
See figure \ref{cap:Ei}. This terminates the definition of $\El_{i+1}$, $G_{i+1}$ and $X_{i+1}$ and, inductively, of the entire sequence.

Most of the properties we require from $(\El_i,G_i,X_i)$ are obvious. In particular clauses (1) and (3) above are clear, as well as the fact that $\El_i\subseteq \El_{i+1}$. As for clause (2), which states that $(G_i,X_i)$ is an exploration process, all properties of an exploration process are clear except the supercriticality \eqref{eq:good>p}. Thus, in the rest of the proof we will demonstrate this property. We do so in several steps.\\

\textit{Step I.} Let $w\sim v$ be two vertices of $\ZZ^2$ and assume $w$ was added to $X$ at step $i_1$ and $v$ at some step $i_2$ for $i_1<i_2$ (that is $w,v\not\in X_{i_1}$, $w\in X_{i_1+1}$, $v\not\in X_{i_2}$ and $w\in X_{i_2+1}$). We note that in this case $\El \cap E_{w,v}$ did not change in the interval $(i_1,i_2]$ namely
  \begin{equation}\label{eq:El stable}
    \El_{i_1+1}\cap E_{w,v}=\El_{i_1+2}\cap E_{w,v}=\dotsb=\El_{i_2}\cap E_{w,v}.
  \end{equation}
  The proof of \eqref{eq:El stable} is straightforward from the construction above and we omit it. 

With this simple geometric fact established we next claim that for each $i$, for every $w\in G_i$ and every $v\in\ZZ^2\setminus X_i$ neighbouring $w$, one has
  \begin{equation}
\mathbb{P}(0\xleftrightarrow{\mathcal{E}_{i}\cup E_{w,v}}M_{v}\,|\,\omega|_{\mathcal{E}_{i}})>1-\delta.\label{eq:goodedge 3}
  \end{equation}
To see this let $w$ be as above, and assume first that $w\neq 0$. This means that for some $1\leq j< i$ we have $w\notin G_j$ but $w\in G_{j+1}$. In other words, $w$ was declared `good' at the $j+1^\textrm{st}$ step. As $v$ neighbours $w$, this implies that \eqref{eq:goodedge 2.5} holds with $\mathcal{D}=\mathcal{D}_{j,w}$, that is,
\[
\mathbb{P}(0\xleftrightarrow{\mathcal{D}_{j,w}\cup E_{w,v}}M_{v}\,|\,\omega|_{\mathcal{D}_{j,w}})>1-\delta.
\]
To obtain \eqref{eq:goodedge 3} we need to show that the $\mathcal{D}_{j,w}$ in the estimate above may be replaced by ${\mathcal{E}_i}$. Using \eqref{eq:El stable} we get that $(\mathcal{D}_{j,w}\cup E_{w,v})\cap \mathcal{E}_{i}= \mathcal{D}_{j,w}$. This means that the conditioning on
$\omega|_{\mathcal{D}_{j,w}}$ in the above inequality may be changed to a conditioning on $\omega|_{\mathcal{E}_i}$ as the additional information in this conditioning is irrelevant for the event. Now, we may replace also the set $\mathcal{D}_{j,w}\cup E_{w,v}$, appearing above the arrow by $\mathcal{E}_{i}\cup E_{w,v}$, as increasing this set will only increase the probability. This completes the proof of \eqref{eq:goodedge 3} for $w\neq 0$.

If $w=0$ we argue as follows. We first note that our definition of $m_2$ implies that
\[
\PP(M_{(0,0)}\xleftrightarrow{Q_{(0,0)}\cup E_{(0,0),v}} M_v)>1-\delta.
\]
Since this is an increasing event, and so is the event that all edges of $Q_{(0,0)}$ are open, they are positively correlated by the FKG inequality. Recalling that we declared the first step to succeeds if all edges of $Q_{(0,0)}$ are open (and that then we have $\El_1=Q_{(0,0)}$), we get
\[
\PP(0\xleftrightarrow{\mathcal{E}_{1}\cup E_{(0,0),v}} M_v\,|\,\omega|_{\mathcal{E}_{1}})>1-\delta.
\]
Again we use \eqref{eq:El stable} to claim that $(\mathcal{E}_{1}\cup E_{w,v})\cap \mathcal{E}_{i}= \mathcal{E}_{1}$, and so the proof may now be completed in much the same way as above (first enlarge the set on which $\omega$ is conditioned, and then  enlarge the set above the arrow).\\

\begin{figure}
  \input{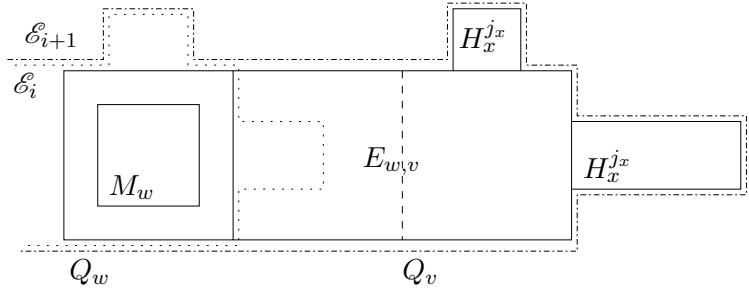}
  \caption{The set $\El_i$ is in a dotted line while $\El_{i+1}$ is dot-dashed. There are two sets denoted $H_x^{j_x}$ in the picture --- they simply relate to different $x$, the right and top ones. There is nothing `wrong' with the fact that no $H$ appear in the picture below --- this would happen if both boxes below were already in $X_i$ when $w$ and $v$ were first explored, or if some $j$ were 0.}\label{cap:Ei}
\end{figure}

\textit{Step II.} Let $v$ be as in \eqref{eq:good>p}. Since $G_i$ and $X_i$ are both functions of $\omega|_{\El_i}$, to obtain \eqref{eq:good>p} it is enough to show that
\begin{equation}\label{eq:bigger pc cond Ei}
\PP(v\in G_{i+1}\,|\,\omega|_{\El_i})>1-\eps
\end{equation}
(recall from the beginning of the lemma that $\eps$ was defined so that $1-\eps>p_c(\ZZ^2,\textrm{site})$). Our first task is to get rid of the conditioning.
We now claim that
our conditioned percolation is equivalent to usual percolation on
an auxiliary graph. Let $\Omega$
be the graph $E_{w,v}\cup\bigcup_{x\in\mathcal{X}}E_{v,x}$
with the following two transformations applied to $\mathcal{E}_{i}\cap E_{w,v}$ (it might help the reader to note that this set is the union of $H^{j}_{w,v}$ for some $1\leq j\leq K-1$ and a single vertex layer, see again figure \ref{cap:Ei}):
\begin{enumerate}
\item For any two vertices $x$ and $y$ in $\mathcal{E}_{i}\cap E_{w,v}$
such that $x\xleftrightarrow{\mathcal{E}_{i}}y$, identify $x$
and $y$.
\item For every edge $e$ with both vertices in $\mathcal{E}_{i}\cap E_{w,v}$
which is closed in $\omega$, remove $e$.
\end{enumerate}
Since $w\in G_{i}$ and $v\not\in X_i$, \eqref{eq:goodedge 3} says that
\[
\mathbb{P}(0\xleftrightarrow{\mathcal{E}_{i}\cup E_{w,v}}M_{v}\,|\,\omega|_{\mathcal{E}_{i}})>1-\delta.
\]
Any connection to $M_v$ through $\El_i\cup E_{w,v}$ must actually use vertices of $E_{w,v}$ (because $v\not\in X_i$ so $Q_v\cap\El_i=\emptyset$, recall \eqref{eq:Ei and Xi}) and in particular $0\xleftrightarrow{\mathcal{E}_{i}}E_{w,v}$, as otherwise the conditioned probability would have been zero.
Hence there is a vertex of $\Omega$ corresponding to $\mathcal{C}(0;\mathcal{E}_{i})\cap E_{w,v}$. 
Call this vertex $o$. Further, (\ref{eq:goodedge 3}) can be represented
in $\Omega$ as follows.
\begin{equation}
\mathbb{P}_\Omega(o\xleftrightarrow{E_{w,v}}M_v)>1-\delta.\label{eq:goodedge from the past}
\end{equation}
(As before, here and below we identify any $A\subset\mathbb{Z}^{d}$ with the
subset of the vertices of $\Omega$ which are the equivalence classes
of all elements of $A$). On the other hand, what we need to prove,
namely that $v$ is good with high probability, can also be represented in $\Omega$. 
We need to show that, with probability at least $1-\varepsilon$,
there exist some $\{j_{x}:x\in\Xl\}$ such that for every
$x\in\mathcal{X}$,
\begin{equation}
\PP_\Omega(o\xleftrightarrow{E_{w,v}\cup E_{v,x}}M_{x}\,|\,\omega|_{E_{w,v}\cup H_{v,x}^{j_x}})>1-\delta.\label{eq:need in Omega}
\end{equation}
All these
equivalences (i.e.\ the facts that we have (\ref{eq:goodedge from the past})
and that we need to show (\ref{eq:need in Omega})) are straightforward
and we omit their proof. From now on we will work only in $\Omega$
and we will no longer note it in the formulas, the way we did in (\ref{eq:goodedge from the past}) or (\ref{eq:need in Omega}).\\

\textit{Step III.} Let $x\in\mathcal{X}$. Recall the definition of $H_{v,x}^j$ above, and let $F_{v,x}^j$ be the face of $H_{v,x}^j$ in the direction of $x$. As usual we give a horizontal example:
\[
F_{v,v+(1,0)}^j=20vr+\{5r+10rj/K\}\times[-2r,2r]^{d-1}.
\]
For  $j\in\{0,\dotsc,K-1\}$ define an auxiliary event $\mathcal{B}_{j}$ by
\[
\Bl_j\coloneqq\{\PP(o\xleftrightarrow{E_{w,v}\cup H_{v,x}}F_{v,x}^{j+1}\,|\,
\omega|_{E_{w,v}\cup H_{v,x}^j}) \le 1-\delta_2\}
\]
(note that the conditioning is up to $j$ but the connection is to $j+1$). We claim that if for some  $j\in\{0,\dotsc,K-1\}$ the event $\Bl_j$ did not happen then the connection between $v$ and $x$ is good, that is,
(\ref{eq:need in Omega}) holds for $v$ and $x$.

To see this we first get rid of the conditioning by defining an auxiliary graph, as usual. The graph is the graph one gets by taking $\Omega$ and identifying any $x,y\in E_{w,v}\cup H_{v,x}^j$ if they are connected in it, and deleting any closed edge both whose vertices are in $E_{w,v}\cup H_{v,x}^j$. Usual percolation on this auxiliary graph is identical to conditioned percolation on $\Omega$. The assumption that $\Bl_j$ did not happen means that in our auxiliary graph $\PP(o\lr F_{v,x}^{j+1})>1-\delta_2$.

  We now apply lemma \ref{lem:vincentless}. We use it with $\eps_{\textrm{lemma \ref{lem:vincentless}}}=\delta$ and the hittable geometry $(\{2K\}\times[-1,1]^{d-1},[-1,2K]\times [-1,1]^{d-1})$ and its image by lattice symmetries. We use the parameters
\[
  D=E_{v,x}\setminus E_{v,x}^j\qquad B=F_{v,x}^{j+1}\qquad T=M_x.
\]
We need to verify that $T$ is a target, namely that for every $u\in \expand{B}{R}$ there is a $\ell(u)$ and a hittable geometry $(F(u),Q(u))$ such that $u+\ell(u)Q(u)\subseteq D$ and $u+\ell(u)F(u)\subseteq T$. The hittable geometry depends only on $x$. Say if $x=v+(0,1)$ then we take $({2K}\times[-1,1]^{d-1},[-1,2K]\times[-1,1]^{d-1})$ and for other $x$ an appropriate rotation of it. The choice of $\ell(u)$ depends only on $j$, we define $\ell(u)=\lfloor (15r-10r(j+1)/K)/2K\rfloor$. Let us verify that $T$ is a target in the horizontal case (the other cases are identical, and assume also $v=0$ for brevity). Since $K\ge 20$ we have $\ell(u)< \frac12 r$ and then
\begin{align*}
  u+\ell(u)F(u)&\subset \Big(5r+\frac{10r(j+1)}{K} + [-R,R]\Big)\times[-2r-R,2r+R]^{d-1}+\\
  &\qquad \Big\lfloor \frac{15r-10r(j+1)/K}{2K}\Big\rfloor\big(\{2K\}\times[-1,1]^{d-1}\big)\\
  &\subset [20r-2K-R,20r+R]\times [-\tfrac52r-R,\tfrac52r+R]^{d-1}.
\end{align*}
Since $r\ge 2KR$ we will indeed have $u+\ell(u)F(u)\subset M_x$. The verification that $u+\ell(u)Q(u)\subset D$ is similar and we omit it.\\

\textit{Step IV.}
It remains to show that with high probability, for every  $x\in\mathcal{X}$ there exists some  $j\in\{0,\dotsc,K-1\}$ so that $\Bl_j$ did not happen. To show this we first claim that
if $x\in\mathcal{X}$ then
\[
\mathbb{P}(o\xleftrightarrow{E_{w,v}\cup H_{v,x}}M_{x})>1-\tfrac{1}{8}\varepsilon.
\]
Indeed, this follows from lemma \ref{lem:shrink} and the parameters we defined. We use lemma \ref{lem:shrink} for the subgraph $E_{w,v}\cup E_{v,x}$ of $\Omega$ with the subbox $D=E_{v,x}\cup Q_v$ (notice that the isomorphism between $D$ and $[-5r,25r]\times[-5r,5r]^{d-1}$ needed for lemma \ref{lem:shrink} might require a rotation) and with $\eps_{\textrm{lemma \ref{lem:shrink}}}=\frac18\eps$. Our assumption that $r>m$, that $o\notin D$ and that $\PP(o\xleftrightarrow{E_{w,v}} M_v)>1-\delta$ \eqref{eq:goodedge from the past} are the conditions of lemma \ref{lem:shrink}, and the conclusion of the lemma is what we need.

Next, denote $\Bl\coloneqq\bigcap_{j=0}^{K-1}\Bl_j$. We claim that
\begin{equation}\label{eq:B cap hit}
\PP(\Bl\cap\{o\xleftrightarrow{E_{w,v}\cup H_{v,x}}M_x\})\le\tfrac18\eps.
\end{equation}
Indeed, to hit $M_x$ via $H_{v,x}$ you must pass through all the $F_{v,x}^j$. However, for each $j$, if $\Bl_j$ happened then there is probability at least $\delta_2$ to not reach $F_{v,x}^{j+1}$, and this bound holds uniformly over whether these events happened for any $j'<j$ or not. So we get that the probability that none of them happened is at most $(1-\delta_2)^K<\frac18\eps$, by our definition of $K$. This shows \eqref{eq:B cap hit}.
Combining the last two estimates we get
\begin{equation}\label{eq:exists Bj}
\PP(\exists j:\neg\Bl_j)\geq \PP(o\xleftrightarrow{E_{w,v}\cup H_{v,x}} M_x\cap \{\exists j:\neg\Bl_j\})>1-\tfrac14\eps.
\end{equation}
The proof is thus finished. Let us retrace our steps to see this. In step III we proved that if $\{\exists j:\neg\Bl_j\}$ then the connection between $v$ and $x$ is good. Thus \eqref{eq:exists Bj} shows that the probability that the connection between $v$ and $x$ is good is at least $1-\frac14\eps$. Summing over $x\in\Xl$ gives that $v$ is good with probability at least $1-\eps$. We proved all this in $\Omega$, but as we explained in step II, the result we proved in $\Omega$ is equivalent to the result in $\ZZ^d$. Thus we get \eqref{eq:bigger pc cond Ei}. We get that $(G_i,X_i)$ satisfies \eqref{eq:good>p} and hence is a supercritical two-dimensional exploration process. This demonstrates clause (\ref{pg:GiXi}) from page \pageref{pg:GiXi}. As explained there, this gives that there
is an infinite cluster in the slab with positive probability (hence
with probability 1), and demonstrates that $p>p_{c}$. 
Since $p$ was an arbitrary value with $\theta(p)>0$, the theorem is proved.
\end{proof}

\section{Miscellenia}\label{sec:misc}

In this section we discuss various related conjectures, possible approaches and metamathematical remarks relating to conjecture \ref{conj:postFKG}, in no particular order.

\subsection{Montone cluster properties}
Let $G$ be a graph and let $f:G\times\{0,1\}^{E(G)}\to\RR$. We say that $f$ is a \emph{monotone cluster property} if it satisfies the following two requirements:
\begin{enumerate}
\item $f(v,\omega)$ depends only on $\Cl_\omega(v)$ the cluster of $v$ in $\omega$, and is increasing in it, i.e.\ if $\Cl_\omega(v)\subseteq\Cl_\xi(v)$ then $f(v,\omega)\le f(v,\xi)$. 
  \item if $\{v,w\}\in\omega$ then $f(v,\omega)=f(w,\omega)$.
\end{enumerate}
(Of course, (2) implies that in fact whenever $v\lr_\omega w$ then $f(v,\omega)=f(w,\omega)$). Examples of monotone cluster properties include $f(v,\omega)=|\Cl_\omega(v)|$ and, for a fixed vertex $b$, $f(v,\omega)=\mathbbm{1}\{b\lr_\omega v\}$.

As far as we know, the following far-reaching generalisation of conjecture \ref{conj:preFKG} might be true.
\begin{conj}\label{conj:mcp}For any graph $G$, any monotone cluster property $f$, any $A\subset G$ and any $0\in G$,
  \[\EE(f(0)\cdot \mathbbm{1}\{0\lr A\})\ge\min_{a\in A}\EE(f(a)\cdot\mathbbm{1}\{0\lr A\}).
  \]
\end{conj}
Conjecture \ref{conj:preFKG} follows from conjecture \ref{conj:mcp} by using the function $f(g)=\mathbbm{1}\{g\lr b\}$. Conjecture \ref{conj:mcp} can be proved in several simple cases.
\begin{thm}\label{thm:mcp A=2}Conjecture \ref{conj:mcp} holds when $|A|=2$, for any $f$.
\end{thm}
\begin{thm}\label{thm:mcp antenna}Conjecture \ref{conj:mcp} holds in any $G$ in which 0 is an isolated vertex in $G\setminus A$, for any $f$.
\end{thm}
\begin{thm}\label{thm:mcp small}Conjecture \ref{conj:mcp} holds for the functions $f(g)=\mathbbm{1}\{|\Cl(g)|\ge k\}$ for $k\le 4$, for any $G$ and $A$.
\end{thm}
We omit the proof of these theorems, as they do not contain any ideas which we did not already use in 
\S\ref{sec:simple cases}. 

Finally we note an interesting connection between different monotone cluster properties in conjecture \ref{conj:mcp}.
\begin{thm} Assume conjecture \ref{conj:mcp} holds for the function $f(g)=|\Cl(g)|$. Then conjecutre \ref{conj:preFKG} holds.
\end{thm}
\begin{proof}[Proof sketch] Let $G$ be a graph, $A\subset G$ and $0,b\in G$. For any $k$ create an auxiliary graph $H_k$ by adding to $G$ $k$ additional vertices and connecting them all to $b$ with edges with weight 1 (and not connecting them to anything else). As $k$ becomes large the size of the cluster becomes a very good approximation of $k\mathbbm{1}\{g\lr b\}$, so using conjecture \ref{conj:mcp} for the given $f$, dividing by $k$ and then taking $k$ to $\infty$ proves conjecture \ref{conj:preFKG}.
\end{proof}
\subsection{Coefficients}\label{sec:coef}
The proof of theorem \ref{thm:A=2} has the curious property that it gives nonnegative coefficients $c_a$ with $\sum c_a=1$ and
\begin{equation}\label{eq:coeffs}
\PP(0\lr b)\ge\sum_{a\in A} c_a\PP(0\lr A, a\lr b),
\end{equation}
and these coefficients do not depend on $b$. It is natural to ask if this property is more general:
\begin{quest}Is it true that for any graph $G$, any $A\subset G$ and any $0\in G$ there are nonnegative coefficients $(c_a:a\in A)$ with $\sum c_a=1$ such that \eqref{eq:coeffs} holds for all $b \in G$?
\end{quest}
If true, this will, of course, imply conjecture \ref{conj:preFKG}. We will now sketch an approach for calculating coefficients of larger $A$ which did not pan out, but we find intriguing. Let us start by considering a random walk analogue.

Assume we are talking about a transient graph (or a graph with a sink, i.e.\ a vertex to which you can go but can never exit). Then the analogue of connecting to $b$ is the expected number of times that a random walk hits $b$ (ever, in the transient case, or until falling into the sink, in the case there is one). Denote this quantity by $G(a,b)$. Thus the question above is approximately analogous to asking whether it is true that
\[
G(0,b)\ge \sum_{a\in A}c_aG(a,b).
\]
The answer is, of course, yes, with $c_a$ being the probability to hit $A$ in $a$.

We now recall a well-known formula that allows to calculate the hitting probabilities $c_a$ using knowledge of $G(0,a)$ and $G(a,a')$ for $a,a'\in A$. Indeed, decomposing $G(0,a)$ according to the first point in $A$ hit along the curve we get
\[
G(0,a)=\sum_{i\in A}c_iG(i,a).
\]
This gives a system of linear equations for the $c_i$.

Going back to percolation, we have one point which we find hard to justify and that is that the analogue of $G(v,w)$ is not the classic $\PP(v\lr w)$ but rather $\PP(0\lr A,v\lr w)$. 
\begin{thm}Let $G$ be a graph and let $A\subset G$ with $|A|=3$. Assume the equations
  \[
  \sum_{a\in A} c_a\PP(0\lr A,a \lr a')=\PP(0\lr a')\qquad\forall a'\in A
  \]
  have a unique solution. Then this solution satisfies $c_a\ge 0$ for all $a\in A$ and $\sum c_a\ge 1$.
\end{thm}
We skip the proof of this theorem as it is long and holds only for $|A|\le 3$. For $|A|\ge 4$ it is possible for some of the $c_i$ to be negative. We only note that it gives an alternative proof of the case $|A|=3$ and $b\in A$ (which is covered by theorem \ref{thm:A=3} (and, in fact, can be easily deduced also from theorem \ref{thm:A=2}). Indeed, assume $b=a_1$. Then
\begin{align*}
\PP(0\lr b)&=\PP(0\lr a_1) = \sum_{a\in A}c_a\PP(0\lr A,a\lr a_1)\\
&\ge \Big(\sum_{a\in A}c_a\Big)\min_{a\in A}\PP(0\lr A,a\lr b)\ge
\min_{a\in A}\PP(0\lr A,a\lr b),
\end{align*}
as needed.

Let us also remark on the question of singularity of this system of equations. Indeed, the classic argument of Aizenman and Newman (see \cite[lemma 3.3]{AN84}) may be adapted to the matrix $\big(\PP(0\lr A,a\lr a')\big)_{a,a'\in A}$ to show that it is always nonnegatively defined, and further that it is positively defined if and only if $\PP(0\lr A)>0$ and $\PP(a\lr a')<1$ for all $a\ne a'$ in $A$. Since the cases where these restrictions are not met are of little interest, we see that the assumption of having a unique solution is harmless.

\subsection{An inductive approach}
\begin{defn*}For a graph $G$ and an edge $e$ denote by $G\Ground e$ the graph where the edge $e$ was shortened. Equivalently, $G \Ground e$ is the graph $G$ with the probability that $e$ is open set to 1.
\end{defn*}
\begin{conj}\label{conj:gluing}Let $G$ be a graph, let $A\subset G$ and $0,b\in G$. Denote by $a$ be the minimiser of $\PP(a\lr b)$ among the points of $A$. Let $\{v,w\}$ be an arbitrary edge and assume
  \begin{equation}\label{eq:works both sides}
  \PP_G(x\lr b)\ge \PP_G(x\lr A)\PP_G(a\lr b)\qquad \forall x\in\{v,w\}.
  \end{equation}
  Then
  \begin{equation}\label{eq:glued}
  \PP_{G\Ground \{v,w\}}(v\lr b)\ge\PP_{G\Ground \{v,w\}}(v\lr A)\PP_{G\Ground \{v,w\}}(a\lr b).
  \end{equation}
\end{conj}
\begin{lem}\label{lem:gluing}Conjecture \ref{conj:gluing} implies the post-FKG conjecture, conjecture \ref{conj:postFKG}.
\end{lem}
\begin{proof}[Proof sketch]We argue by induction on the number of edges in the graph. Assume conjecture \ref{conj:postFKG} has been shown for all graphs with less than $n$ edges. Let $G$ be some graph with $n$ edges, and let $e$ be an edge of $G$. Let $H_p$ be the auxiliary graph which one gets from $G$ by changing the probability of the edge $\{v,w\}$ to be open to be $p$ (so $H_1 =G\Ground \{v,w\}$). Let $a\in A$ be the one minimising $\PP_{H_0}(a\lr b)$ among the points of $A$ and define
  \[
  f(p)\coloneqq \PP_{H_p}(v\lr b)-\PP_{H_p}(v\lr A)\PP_{H_p}(a\lr b).
  \]
In words, the function $f$ measures the effect of the single edge $\{v,w\}$ on the probabilities of interest.

By our induction assumption, conjecture \ref{conj:postFKG} holds for $H_0$, i.e.\ $f(0)\ge 0$. Further, our induction assumptions shows \eqref{eq:works both sides}, so by conjecture \ref{conj:gluing} (for $H_0$), $f(1)\ge 0$. Finally, since $f(p)$ is concave in $p$, if it is nonnegative at 0 and at 1, it must be nonnegative for all $p$. This shows conjecture \ref{conj:postFKG} for $G$ and terminates the induction.
\end{proof}

\subsection{An application of the theory of boolean functions}

In this section we show a result that, says, roughly, that if graph does not satisfy conjecture \ref{conj:postFKG} but a lot of similar graphs do satisfy it, then it can only break the conjecture by a small amount. This result is similar to the well-known fact that for a monotone boolean function on $n$ bits the total influence is bounded by $\sqrt{n}$.

\begin{thm}
Fix $0,b\in G$ and $A\subset G$ and let $E$ be the
set of edges for which
\[
\mathbb{P}_{G\Ground e}(0\leftrightarrow b)\ge\mathbb{P}_{G\Ground e}(0\leftrightarrow A)\min_{a\in A}\mathbb{P}_{G\Ground e}(a\leftrightarrow b).
\]
Then
\[
\mathbb{P}_{G}(0\leftrightarrow b)\ge\mathbb{P}_{G}(0\leftrightarrow A)\min_{a\in A}\mathbb{P}_{G}(a\leftrightarrow b)-\bigg(\sum_{e\in E}\frac{p_{e}}{1-p_{e}}\bigg)^{-1/2}.
\]
\end{thm}

\begin{proof}
Denote
\[
m\coloneqq\mathbb{P}_{G}(0\leftrightarrow A)\min_{a\in A}\mathbb{P}_{G}(a\leftrightarrow b)-\mathbb{P}_{G}(0\leftrightarrow b)
\]
and assume $m>0$ (otherwise there is nothing to prove). Fix an $e\in E$
and let $G_{p}$ be the graph $G$ with the probability of the edge
$e$ changed to $p$ (so in particular $G_{1}=G\Ground e$ and $G_{p_{e}}=G$).
Let $a\in A$ be the point where the minimum is achieved at $p=1$.
Let $f:[0,1]\to\mathbb{R}$ be defined by
\[
f(p)=\mathbb{P}_{G_{p}}(0\leftrightarrow b)-\mathbb{P}_{G_{p}}(0\leftrightarrow A)\mathbb{P}_{G_{p}}(a\leftrightarrow b).
\]
By definition $f(1)\ge0$ and $f(p_{e})=-m$. Since $f$ is continuously
differentiable there must be some $\xi\in[p_{e},1]$ such that
\[
f'(\xi)\geq \frac{m}{1-p_{e}}.
\]
The derivative of the second term in the definition of $f$ is negative,
so we get
\[
\left.\frac{d}{dp}\mathbb{P}_{G_{p}}(0\leftrightarrow b)\right|_{\xi}\ge\frac{m}{1-p_{e}}.
\]
Now the derivative on the left hand side no longer depends on $\xi$ and, further, can be written in the following
way. Define a function on $\{0,1\}^{E(G)}$ by
\[
D_{e}(\omega)\coloneqq\begin{cases}
\frac{1}{p_{e}} & \omega_{e}=1\\
\frac{-1}{1-p_{e}} & \omega_{e}=0.
\end{cases}
\]
Then one can check that
\[
\frac{d}{dp}\mathbb{P}_{G_{p}}(0\leftrightarrow b)=\mathbb{E}(\mathbbm{1}\{0\leftrightarrow b\}D_{e}).
\]
(the expectation on the right is simply on $G$). Now, the functions
$D_{e}$ for different $e$ are orthogonal. Hence
\begin{align*}
1 & \ge\mathbb{E}(\mathbbm{1}\{0\leftrightarrow b\}^{2})\ge\sum_{e\in E}\frac{\mathbb{E}(\mathbbm{1}\{0\leftrightarrow b\}D_{e})^{2}}{\mathbb{E}(D_{e}^{2})}\ge\sum_{e\in E}\frac{\left(\frac{m}{1-p_{e}}\right)^{2}}{\frac{1}{p_{e}(1-p_{e})}}\\
 & =m^{2}\sum_{e\in E}\frac{p_{e}}{1-p_{e}}
\end{align*}
as promised.
\end{proof}

\begin{rem*} It is easy to check that a similar result holds in the pre-FKG case (corresponding to conjecture \ref{conj:preFKG}).
\end{rem*}

\subsection{Candidates}

We assemble here a collection of attempts to quantify the question `for which $a$ does the conjecture hold except the minimal one?' Numerical evidence points towards an answer of `yes' for all of them, but we are rather hesitant.
\begin{quest}Let $G$ be a graph, $A\subset G$, $0,b\in G$. Let $a\in A$ satisfy
  \[
  \PP(a\lr b)\le \PP(a'\lr b)\qquad\forall a'\in A,
  \]
  i.e.\ $a$ is the minimiser among the points of $A$. Is it true that in this case
  \begin{equation}\label{eq:it is a}
  \PP(0\lr b,0\lr A)\ge \PP(0\lr A,a\lr b)?
  \end{equation}
  In other words, does the minimiser of the post-FKG version of the conjecture also satisfy the pre-FKG version (even though it is not necessarily the minimiser in this second case)?
\end{quest}
\begin{quest}
  Let $G$, $A$, $0$ and $b$ be as above. Let $a$ be the point minimising $\PP(a\lr b,0\nlr A)$ among the points of $A$. Is it true that in this case the pre-FKG conjecture \eqref{eq:it is a} holds?
\end{quest}
\begin{quest}Let $G$, $A$, $0$ and $b$ be as above. Let $H$ be the graph given from $G$ by removing all edges going out of $0$. Let $a$ be the point minimising $\PP_H(a\lr b)$ among the points of $A$. Is it true that in this case \eqref{eq:it is a} holds?
\end{quest}

\subsection{Simple reductions}
We note three simple reductions that can be performed on the problem.
\lazyenum
\item It is possible to assume that all edges have probability $\frac12$. Indeed, for any $p\in[0,1]$ and any $\eps>0$ one may find a graph $H_p$, all whose edges have probability $\frac12$, and two vertices $a,b\in H_p$ such that $\PP_{H_p}(a\lr b)\in (p-\eps,p+\eps)$. (The graph $H_p$ can be constructed simply by taking many parallel lines between $a$ and $b$). Thus for any finite graph $G$ with any probabilities one can simply replace each edge $e$ with weight $p$ with a copy of $H_p$ with the $a$ and $b$ of $H_p$ identified with the two vertices of $e$. Taking $\eps\to 0$ proves the claim. 

  This allows to reformulate the problem as a counting problem.
\item It is possible to assume all vertices of $G$ have degree 3. Indeed, any vertex of degree higher than 3 can be simulated by a little graph where all vertices have degree 3 and all internal edges have probability 1. Vertices with degree smaller than 3 may be decorated by small $
  \begin{tikzpicture}[baseline]
    \fill[black] (0.2,0.1) circle [radius=1pt];
    \fill[black] (0.3,0) circle [radius=1pt];
    \fill[black] (0.5,0) circle [radius=1pt];
    \fill[black] (0.5,0.2) circle [radius=1pt];
    \fill[black] (0.3,0.2) circle [radius=1pt];
    \draw (0,0.1) -- (0.2,0.1) -- (0.3,0) -- (0.5,0) -- (0.5,0.2) -- (0.3,0.2) -- (0.2,0.1);
    \draw (0.5,0) -- (0.3,0.2);
    \draw (0.3,0) -- (0.5,0.2);
  \end{tikzpicture}$, 1 for each missing edge. Such dangling ends do not change connection probabilities.
  This might be useful for an inductive approach to the problem.

  It would be interesting to combine these two reductions, but it seems that this would require a new idea.

\item It is possible to assume that $0\not\sim b$. Indeed, let $G$ be some arbitrary graph, and denote by $G_p$ the graph where the edge $\{0,b\}$ is set to have probability $p$. Our assumption is that the conjecture holds for $G_0$ (for some $a\in A$). At $G_1$, though, the conjecture holds trivially, for all $a\in A$. Using the fact that the value we wish to prove is positive is concave (in the post-FKG version) or linear (in the pre-FKG version) it must be positive throughout $[0,1]$.

\item It is possible to restrict to sets $A$ of the sort $\{a\in G: \PP(a\lr b)>p\}$ for some $p\in [0,1]$. 
\end{enumerate}

\subsection{Counterexamples}
\begin{figure}
  \centering
  \input{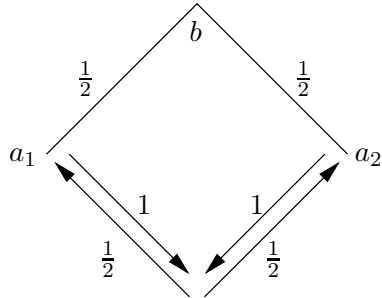}
  \caption{A directed percolation counterexample.}\label{cap:directed}
\end{figure}
The most obvious wrong generalisation of conjecture \ref{conj:postFKG} is that it does not work for directed percolation. Indeed, consider a graph with 4 vertices, 0, $a_1$, $a_2$ and $b$ with directed edges (with probability $\frac12$) from $0$ to both $a_i$, and from those to $b$. In addition we put directed edges from the $a_i$ back to $0$ with probability 1. See figure \ref{cap:directed}. We have $\PP(0\to b)=\frac7{16}$ while $\PP(0\to A)\PP(a\to b)=\frac34\cdot\frac58>\frac7{16}$. 

Conjecture \ref{conj:gluing} might tempt one to conjecture that in fact, for any $a$ for which $\PP_G(v\lr b)\ge \PP_G(v\lr A)\PP_G(a\lr b)$ one has \eqref{eq:glued} (with $G$ and $v$ as in conjecture \ref{conj:gluing}). If true, this would reduce conjecture \ref{conj:postFKG} to a question involving a bounded number of points. Rather than giving a counterexample (one exists already in a graph with 5 vertices and $|A|=3$), we give an indirect argument which is more generally applicable.

\begin{thm}
  There exists a graph $G$, a set $A\subset G$, and vertices $a\in A$, $b,v,w\in G$ such that
  \[
  \PP_G(v\lr b)\ge \PP_G(\textcolor{blue}{v}\lr A)\PP_G(a\lr b)
  \]
  but
  \[
  \PP_{G\Ground \{v,w\}}(v\lr b)<\PP_{G\Ground \{v,w\}}(v\lr A)\PP_{G\Ground \{v,w\}}(a\lr b).
  \]
\end{thm}
\begin{proof}Assume by contradiction that the opposite holds for all $G$, $A$, $a,b,v,w$. Let $G$ be a graph where $v$ is disconnected from $A$. In this graph
\begin{equation}\label{eq:nubeemet}
  \PP_G(v\lr b)\ge \PP_G(v\lr A)\PP_G(a\lr b)
\end{equation}holds for all $a\in A$ because the right hand side is 0. Applying the contradictory assumption, the same argument as in lemma \ref{lem:gluing} shows that in any graph one can get from $G$ by successively adding edges we will also have \eqref{eq:nubeemet} for all $a\in A$. Since any graph can be had from a graph where $v$ and $A$ are disconnected by adding edges, we get that in any graph $G$,  any $A$ and any $a\in A$ we have \eqref{eq:nubeemet}. This is clearly false.
\end{proof}


\begin{thebibliography}{10}

\bibitem {AKN}  Michael Aizenman, Harry Kesten and Charles M. Newman, \emph{Uniqueness of the infinite cluster and continuity of connectivity functions for short and long range percolation}. Comm. Math. Phys. 111 (1987), no. 4, 505--531. \afs{https://projecteuclid.org/journals/communications-in-mathematical-physics/volume-111/issue-4/Uniqueness-of-the-infinite-cluster-and-continuity-of-connectivity-functions/cmp/1104159720.full}{projecteuclid.org/159720.full}
  
\bibitem{AN84}David J. Barsky and Michael Aizenman, \emph{Percolation critical exponents under the triangle condition}. Ann. Probab. 19:4 (1991), 1520--1536. Available: \afs{https://www.jstor.org/stable/2244525}{jstor.org/2244525}

\bibitem{BGN} David J. Barsky, Geoffrey R. Grimmett and Charles M. Newman, \emph{Percolation in half-spaces: equality of critical densities and continuity of the percolation probability}. Probab. Theory Related Fields 90:1 (1991), 111--148.  \affs{https://doi.org/10.1007/BF01321136}{springer.com/BF01321136}

\bibitem{BLPS} Itai Benjamini, Russell Lyons, Yuval Peres and Oded Schramm, \emph{Critical percolation on any nonamenable group has no infinite clusters}. Ann. Probab. 27:3 (1999), 1347--1356. Available at: \affs{https://doi.org/10.1214/aop/1022677450}{projecteuclid.org/1022677450}

\bibitem{vdBHK}
Jacob van den Berg, Olle H\"aggstr\"om and Jeff Kahn, \textit{Some conditional correlation inequalities for percolation and related processes,} Random Struct.\ Alg.\ 29:4 (2006), 417--435. Available at: \affs{https://doi.org/10.1002/rsa.20102}{wiley.com/rsa.20102}

\bibitem{BK}  Robert M. Burton Jr. and Michael S. Keane, \emph{Density and uniqueness in percolation}. Comm. Math. Phys. 121 (1989), no. 3, 501--505. \afs{https://projecteuclid.org/journals/communications-in-mathematical-physics/volume-121/issue-3/Density-and-uniqueness-in-percolation/cmp/1104178143.full?tab=ArticleLink}{projecteuclid.org/178143}
  
\bibitem{Cerf} Rapha\"el Cerf, \emph{A lower bound on the two-arms exponent for critical percolation on the lattice}. Ann. Probab. 43:5 (2015), 2458--2480. \afs{https://projecteuclid.org/journals/annals-of-probability/volume-43/issue-5/A-lower-bound-on-the-two-arms-exponent-for-critical/10.1214/14-AOP940.full}{projecteuclid/14-AOP940.full}

\bibitem{Engel}Diederik G. P. van Engelenburg, \emph{Bounds on the two-arms probability for percolation on $\mathbb{Z}^d$}. MSc thesis, University of Amsterdam, 2020. Available at: \href{https://scripties.uba.uva.nl/search?id=record_26816}{\nolinkurl{uva.nl/26816}}

\bibitem{FvdH17}  Robert Fitzner and Remco van der Hofstad, \emph{Mean-field behavior for nearest-neighbor percolation in $d>10$}. Electron. J. Probab. 22 (2017), Paper No. 43, 65 pp. Available at: \affs{https://projecteuclid.org/journals/electronic-journal-of-probability/volume-22/issue-none/Mean-field-behavior-for-nearest-neighbor-percolation-in-d10/10.1214/17-EJP56.full}{projecteuclid/10.1214}

\bibitem{Gri}
  Geoffrey R. Grimmett, \emph{Percolation}. Second edition. Grundlehren der mathematischen Wissenschaften [Fundamental Principles of Mathematical Sciences], 321. Springer-Verlag, Berlin, 1999.

\bibitem{GM} Geoffrey R. Grimmett and John M. Marstrand, \emph{The supercritical phase of percolation is well behaved}. Proc. Roy. Soc. London Ser. A 430 (1990), no. 1879, 439--457. Available at: \affs{https://doi.org/10.1098/rspa.1990.0100}{royalsocietypublishing.org/1990.0100}

\bibitem{H90} Takashi Hara, \emph{Mean-field critical behaviour for correlation length for percolation in high dimensions.} Probab. Theory Related Fields 86:3 (1990), 337--385. Available at: \affs{https://doi.org/10.1007/BF01208256}{springer.com/BF01208256}

\bibitem{H60}  Theodore E. Harris, \emph{A lower bound for the critical probability in a certain percolation process}. Proc. Cambridge Philos. Soc. 56 (1960), 13--20.

\bibitem{HH} Jonathan Hermon and Tom Hutchcroft, \emph{No percolation at criticality on certain groups of intermediate growth}. Int. Math. Res. Not. IMRN 2021:22, 17433--17455. Available at: \affs{https://doi.org/10.1093/imrn/rnz265}{oup.com/rnz265}

\bibitem{H} Tom Hutchcroft, \emph{Critical percolation on any quasi-transitive graph of exponential growth has no infinite clusters}. C. R. Math. Acad. Sci. Paris 354:9 (2016), 944--947. Available at: \affs{https://doi.org/10.1016/j.crma.2016.07.013}{sciencedirect.com/2016.07.013}

\bibitem{K80} Harry Kesten, \emph{The critical probability of bond percolation on the square lattice equals $\frac12$}. Comm. Math. Phys. 74:1 (1980), 41--59. Available at: \affs{https://projecteuclid.org/journals/communications-in-mathematical-physics/volume-74/issue-1/The-critical-probability-of-bond-percolation-on-the-square-lattice/cmp/1103907931.full?tab=ArticleLink}{projecteuclid.org/3907931}

\bibitem{N86}  Charles M. Newman, \emph{Some critical exponent inequalities for percolation}. J. Statist. Phys. 45:3--4 (1986), 359--368. Available at: \affs{https://doi.org/10.1007/BF01021076}{springer.com/BF01021076}

\end{thebibliography}
\end{document}